\documentclass[leqno,a4paper,12pt]{article}
\usepackage[english]{babel}
\usepackage[utf8x]{inputenc} 
\usepackage[T1]{fontenc} 
\usepackage{hyperref} 
\usepackage{fancyhdr} 
\usepackage{pgf,tikz} 
\usepackage{float} 
\usetikzlibrary{arrows} 

\usepackage{amsmath, amsfonts, amssymb, amsthm}

\usepackage{graphicx,color,epsfig}
\usepackage{euscript}
\usepackage{mathrsfs}

\usepackage{stmaryrd} 
\newcommand{\brint}[1]{{\llbracket #1 \rrbracket}} 

\allowdisplaybreaks[1]

\newtheorem{theorem}{Theorem}
\newtheorem{definition}{Definition}
\newtheorem{lemma}{Lemma}

\newtheorem{proposition}{Proposition}

\makeatletter

\@addtoreset{equation}{section}
\makeatother

\def\R{{\mathbb R}} 
\def\N{{\mathbb N}} 
\def\Ns{{\mathbb N}^*} 

\def\P{{\bf P}} 
\def\Ph{{\widehat {\bf P}}} 
\def\E{{\bf E}} 
\def\Eh{{\widehat{\bf E}}} 
\def\iid{{i.i.d.\ }} 

\def\root{{\rho}} 
\def\Tb{{\mathbf T}} 
\def\Fb{{\mathbf F}} 
\def\X{{\cal{X}}} 
\def\T{{\mathbb T}} 
\def\Tbh{{\widehat{\mathbf T}}} 
\def\F{{\mathbb F}} 
\def\L{{\cal{L}}} 
\def\B{{\cal{B}}}
\def\Zi{L} 
\def\Ni{B} 
\def\tauh{{\hat{\tau}}} 
\def\ty{{\bf e}} 

\newcommand{\cl}[1]{{\lceil #1 \rceil}} 
\newcommand{\fl}[1]{{\lfloor #1 \rfloor}} 
\newcommand{\1}[1]{{{\bf 1}_{\{ {#1} \}}}} 
\newcommand{\sto}[2][\longrightarrow]{{\substack{ \\ {#1} \\ {#2}}}} 
\newcommand{\parent}[1]{{\buildrel \leftarrow \over {#1}}} 
\newcommand{\up}[1]{\textsuperscript{#1}} 
\def\eps{\varepsilon} 
\def\d{\, \mathrm{d}} 
\renewcommand{\cal}[1]{{\mathcal{#1}}} 
\newcommand{\scr}[1]{{\mathscr{#1}}} 
\newcommand{\st}{{\; :\;}} 
\newcommand{\commentblock}[1]{} 

\def\bzeta{{\boldsymbol \zeta}} 
\def\hzeta{{\widehat{\zeta}}} 

\title{\textbf{Scaling limit of multitype Galton--Watson trees with infinitely many types}}
\author{Loïc de Raphélis \thanks{Sorbonne Universités, UPMC Univ Paris 6, UMR 7599, Laboratoire de Probabilités et Modèles Aléatoires, 4 place
Jussieu, F-75005 Paris. {\tt Loic.de\_raphelis\_soissan$@$upmc.fr}}}
\date{\today}

\begin{document}

\maketitle

\begin{abstract}
   We introduce a certain class of 2-type Galton--Watson trees with edge lengths. We prove that, after an adequate rescaling, the weighted height function of a forest of such trees converges in law to the reflected Brownian motion. We then use this to deduce under mild conditions an invariance principle for multitype Galton--Watson trees with a countable number of types, thus extending a result of G.~Miermont in~\cite{miermont} on multitype Galton--Watson trees with finitely many types. 
\end{abstract}

\bigskip

\noindent{\bf MSC2010 subject classification }60J80, 60F17

\bigskip

\noindent{\bf Key Words~: } Galton--Watson tree, Multitype Galton--Watson tree, infinitely many types, edge lengths, scaling limit
\clearpage

\section{Introduction}

In a seminal work~\cite{aldous3}, D.~Aldous established the scaling limit of critical Galton--Watson trees with finite variance conditioned to be large as the continuum random tree~\cite{aldous1,aldous2}. One way to see that is to consider the {\it height functions} of the trees and to show that the latter converge towards the Brownian excursion. Later on, T.~Duquesne and J.-F.~Le Gall~\cite{duquesne-le-gall} showed the convergence in law of the height function of the critical Galton--Watson forest with possibly infinite variance towards a Lévy Process. 

This result on Galton--Watson forests with finite variance was extended by G.~Miermont in~\cite{miermont} to critical multitype Galton--Watson trees with finitely many types, under a second moment condition. Our aim is to get this result when the set of types is countable, under mild conditions. 

To this end, we will introduce first a certain kind of 2-type Galton--Watson trees with edge lengths, one of the types being sterile, that we will call {\it leafed Galton--Watson trees with edge lengths}, as the vertices of sterile type can be seen as extra leaves. We will prove that under certain hypotheses, the height function of a forest made up of such trees, taking into account the edge lengths, satisfies the same limit theorem than simple Galton--Watson forests with finite variance in~\cite{duquesne-le-gall}. This result will then be used to prove the convergence of multitype Galton--Watson forests : our method will consist in linking the height function of any given multitype Galton--Watson tree to that of a certain leafed Galton--Watson tree with edge lengths, using a tree-reduction method inspired by Section 2.3 of~\cite{miermont}. 

Several results on Galton--Watson trees with edge lengths have already been obtained. In~\cite{durrett-kesten-waymire} R.~Durrett, H.~Kesten and E.~Waymire determined among others  the asymptotic distribution of the maximal weighted height of a Galton--Watson tree conditioned on total progeny, when edge lengths are \iid Then, M.~Ossiander, E.~Waymire and Q.~Zhang proved in~\cite{ossiander-waymire-zhang} the convergence in law of the weighted height function of critical Galton--Watson trees conditioned on total progeny to the Brownian excursion, still in the case of \iid edge lengths.

Moreover, leafed Galton--Watson trees with edge lengths will find another application in an upcoming paper~\cite{aidekon-de-raphelis} where we will show how the study of a random walk on a Galton--Watson tree can be reduced to that of the height process of a leafed Galton--Watson forest with edge lengths. \\

\subsection{Leafed Galton--Watson trees with edge lengths}\label{s:introleafed}

We consider a random tree the vertices of which may be of type $0$ or $1$, and the edges of which have random lengths. Types $0$ and $1$ are such that a vertex may have a progeny only if its type is $1$. More precisely, our process will consist in a triplet $(\T,e,\ell)$ where for any $u$ in the tree $\T$, $e(u)$ is the type of $u$ and $\ell(u)$ is a non-negative number standing for the length of the edge joining $u$ with its parent. Let $\zeta$ be a probability measure on $\bigcup_{n≥ 0}(\{0;1\}\times \R_+)^n$ (with the convention that $(\{0;1\}\times \R_+)^0$ is the empty sequence) ; we call $\zeta$ the \textit{offspring distribution}. Notice that realisations of $\zeta$ are ordered, as we will only consider \textit{planar trees}.  We construct $(\T,e,\ell)$ by induction on generations as follows~: 
\begin{itemize}
\item {\bf Initialisation} \\
Generation $0$ of $\T$ is only made up of the root, denoted by $\root$, such that $e(\root)=1$ and $\ell(\root)=0$. 
\item {\bf Induction} \\
Let $n≥0$, and suppose that the tree has been built up to generation $n$. If generation $n$ is empty, then generation $n+1$ is empty. Otherwise, each vertex $u$ of generation $n$ such that $e(u)=1$ gives progeny according to $\zeta$, independently of other vertices, thus forming generation $n+1$. Vertices $u$ of generation $n$ such that $e(u)=0$ give no progeny. 
\end{itemize}
We call $(\T,e,\ell)$ a {\it leafed Galton--Watson tree with edge lengths}. We denote its law by $\P$, and by $\E$ the associated expectation. Notice that the subset of vertices of type $1$ has the law of a Galton--Watson tree : we denote it by $\T^1$, and we let $\zeta^1$ be its reproduction law (which includes the information on $\ell$). The tree $\T$ can therefore be seen as the tree $\T^1$ to which leaves (the vertices of type $0$) were artificially added (hence the word {\it leafed}). Likewise, we can define a \textit{leafed Galton--Watson forest with edge lengths} $(\F,e,\ell)$ as a sequence of \iid leafed Galton--Watson trees with edge lengths~; we denote by $\F^1$ the subset of vertices of type $1$ of $\F$.  \\
\bigskip

We will code planar trees using Neveu's notation~\cite{neveu}. Let $\cal{U}:=\bigsqcup_{n≥1}{(\N^*)}^{n}\cup\{ \root \}$ be the infinite Ulam-Harris tree. This tree is the set of all possible vertices. For $u,v\in\cal{U}$, we let $u.v$ be the concatenation of the sequences $u$ and $v$ (with $u.\root=\root.u=u$). This coding can be extended to forests : if $\F$ is a forest made up of trees $\T_1,\T_2,\ldots$, and if $u\in\T_i$, then we will code it by $(i).u$ in $\F$. With this notation, the roots of $\F$ are denoted by $(1),(2),\ldots$ (so $\root\notin\F)$. 

For any vertices $u$, $v$ in the tree $\T$, we let
\begin{itemize}
\item $|u|$ be the generation of $u$ (the root $\root$ being at generation $0$), 
\item $u_0,u_1,\ldots,u_{|u|}$ be the ancestors of $u$ at generation $0,1,2,\ldots,|u|$, 
\item $\Omega(u)$ be the set of its brothers (that is vertices $v\neq u$ in $\T$ having the same parent), 
\item $\nu(u)$ be its number of children in $\T$, and $\nu$ be a generic random variable with same law than $\nu(\root)$,
\item $\nu^1(u)$ be its number of children of type 1 in $\T$ (that is its number of children in $\T^1$), and $\nu^1$ be a generic random variable with same law than $\nu^1(\root)$,
\item $\parent{u}$ be its parent, 
\item $u \vdash v$ if $u$ is a strict ancestor of $v$, that is if there exists $w\in\cal{U}\setminus \{\root\}$ such that $v=u.w$,
\item $u \prec v$ if $u$ is lexicographically strictly smaller than $v$,
\item $u(0)=\root, u(1), u(2), \ldots$ be the vertices of $\T$ ordered lexicographically (if $\T$ is finite), 
\item $u^1(0)=\root, u^1(1), u^1(2), \ldots$ be the vertices of $\T$ of type $1$ ordered lexicographically (that is they are the vertices of $\T^1$ ordered lexicographically) (if $\T^1$ is finite). 
\end{itemize}
\noindent Notice that this notation can be naturally extended to $\F$, with the convention that roots $(1),(2),\ldots$ in $\F$ are at generation $0$. \\

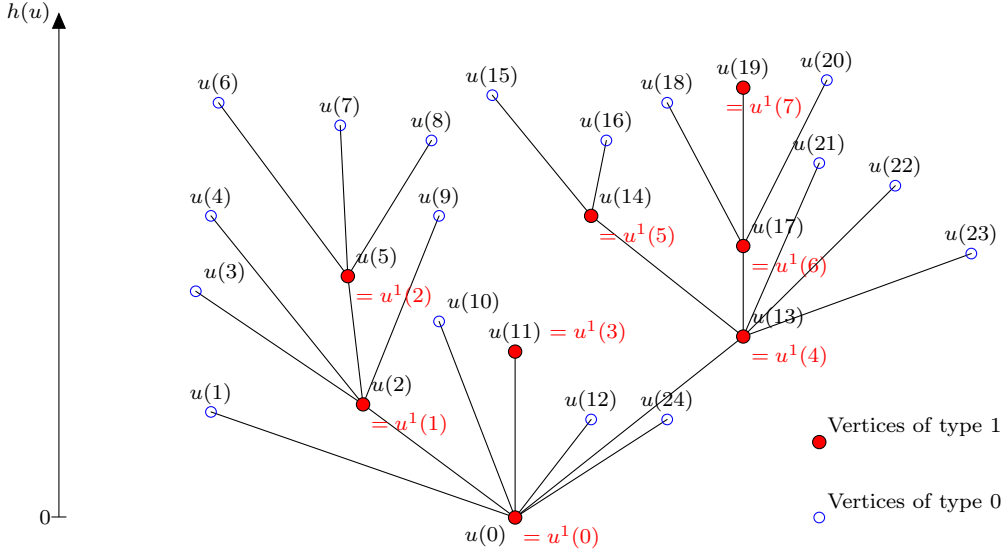
\begin{figure}[H]
\begin{tikzpicture}[line cap=round,line join=round,>=triangle 45,x=1.0cm,y=1.0cm]
\clip(-0.1,-0.5) rectangle (13.5,7);
\begin{scriptsize}
\draw (7,0) node[below left] {$u(0)$} node[below right, color=red] {$=u^1(0)$};
\draw (7,0)-- (3,1.4) node[above]{$u(1)$};
\draw (7,0)-- (5,1.5) node[above right]{$u(2)$} node[below right,color=red]{$=u^1(1)$};
\draw (7,0)-- (6,2.6) node[above right]{$u(10)$};
\draw (7,0)-- (7,2.2) node[above]{$u(11)$} node[above right,color=red]{$\phantom{1)}=u^1(3)$};
\draw (7,0)-- (8,1.3) node[above]{$u(12)$};
\draw (7,0)-- (9,1.3) node[above]{$u(24)$};
\draw (7,0)-- (10,2.4) node[above right]{$u(13)$} node[below right,color=red]{$=u^1(4)$};
\draw (5,1.5)-- (3,4) node[above]{$u(4)$};
\draw (5,1.5)-- (6,4) node[above]{$u(9)$};
\draw (5,1.5)-- (4.8,3.2)node[above right]{$u(5)$} node[below right,color=red]{$=u^1(2)$};
\draw (5,1.5)-- (2.8,3) node[above right]{$u(3)$};
\draw (10,2.4)-- (8,4) node[above right]{$u(14)$} node[below right,color=red]{$=u^1(5)$};
\draw (10,2.4)-- (10,3.6) node[above right]{$u(17)$} node[below right,color=red]{$=u^1(6)$};
\draw (10,2.4)-- (11,4.7) node[above]{$u(21)$};
\draw (10,2.4)-- (12,4.4) node[above]{$u(22)$};
\draw (10,2.4)-- (13,3.5) node[above]{$u(23)$};
\draw (4.8,3.2)-- (3.1,5.5) node[above]{$u(6)$};
\draw (4.8,3.2)-- (4.7,5.2) node[above]{$u(7)$};
\draw (4.8,3.2)-- (5.9,5) node[above]{$u(8)$};
\draw (10,3.6)-- (9,5.5) node[above]{$u(18)$};
\draw (10,3.6)-- (10,5.7) node[above]{$u(19)$} node[below,color=red]{$\phantom{aaa}=u^1(7)$};
\draw (10,3.6)-- (11.1,5.8) node[above]{$u(20)$};
\draw (8,4)-- (6.7,5.6) node[above]{$u(15)$};
\draw (8,4)-- (8.2,5) node[above]{$u(16)$};
\draw[|->] (1,0)-- (1,6.7);
\draw (1,0) node[left] {$0$};
\draw (1,6.7) node[left] {$h(u)$};
\draw (11,1.2) node[anchor= west] {Vertices of type 1};
\draw (11,0.2) node[anchor= west] {Vertices of type 0};
\draw [fill=red] (7,0) circle (2.5pt);
\draw [color=blue] (3,1.4) circle (2.0pt);
\draw [fill=red] (5,1.5) circle (2.5pt);
\draw [color=blue] (6,2.6) circle (2.0pt);
\draw [fill=red] (7,2.2) circle (2.5pt);
\draw [color=blue] (8,1.3) circle (2.0pt);
\draw [color=blue] (9,1.3) circle (2.0pt);
\draw [fill=red] (10,2.4) circle (2.5pt);
\draw [color=blue] (3,4) circle (2.0pt);
\draw [color=blue] (6,4) circle (2.0pt);
\draw [fill=red] (4.8,3.2) circle (2.5pt);
\draw [color=blue] (2.8,3) circle (2.0pt);
\draw [fill=red] (8,4) circle (2.5pt);
\draw [fill=red] (10,3.6) circle (2.5pt);
\draw [color=blue] (11,4.7) circle (2.0pt);
\draw [color=blue] (12,4.4) circle (2.0pt);
\draw [color=blue] (13,3.5) circle (2.0pt);
\draw [color=blue] (3.1,5.5) circle (2.0pt);
\draw [color=blue] (4.7,5.2) circle (2.0pt);
\draw [color=blue] (5.9,5) circle (2.0pt);
\draw [color=blue] (9,5.5) circle (2.0pt);
\draw [fill=red] (10,5.7) circle (2.5pt);
\draw [color=blue] (11.1,5.8) circle (2.0pt);
\draw [color=blue] (6.7,5.6) circle (2.0pt);
\draw [color=blue] (8.2,5) circle (2.0pt);
\draw [fill=red] (11,1) circle (2.5pt);
\draw [color=blue] (11,0) circle (2.0pt);
\end{scriptsize}
\end{tikzpicture}
\caption{An example of a realisation of a leafed Galton--Watson tree with edge lengths. }
\label{f:edge}
\end{figure}
\medskip

\noindent We make the following hypotheses~: 
\begin{itemize}
\item[$(\bf H_1)$ :] $\E\Big[  \sum_{|u|=1} 1 \Big]=\E\Big[ \nu \Big]=:m<\infty$,
\item[$(\bf H_c)$ :] $\E\Big[  \sum_{|u|=1} \1{e(u)=1} \Big]=\E\Big[ \nu^1 \Big]=1$,
\item[$(\bf H_c^2)$ :] ${\bf Var}\Big(  \sum_{|u|=1} \1{e(u)=1} \Big)={\bf Var}\Big( \nu^1 \Big)=:\sigma^2\in(0;\infty)$ (we take $\sigma>0$), 
\item[$(\bf H_0^2)$ :] $y^2\P\Big(  \max_{|u|=1,e(u)=0} \ell(u)>y \Big) \sto{y\to\infty} 0$,
\item [$(\bf H_1^2)$ :] $y^2\E\Big[\sum_{|u|=1, e(u)=1}\1{\ell(u)>y} \Big]\sto{y\to\infty} 0$,
\end{itemize}
\noindent and we denote by $\bf(H)$ their union. The second and third hypotheses ensure that $\T^1$ is a non-trivial critical Galton--Watson tree with finite variance. Notice that under hypotheses $\bf (H_1)$ and $\bf (H_c)$, $\T$ is finite, thus making consistent the numbering $u(0),u(1),\ldots$ previously introduced in the notation.

We let
\begin{equation*}
\mu:=\E\Big[\sum_{|u|=1,e(u)=1} \ell(u) \Big]
\end{equation*}
be the mean of the sum of lengths of edges issued from vertices of type 1, which is finite thanks to $\bf (H_1^2)$. For each vertex $u\in\F$, we define its height $h(u)$ as~: 
\begin{equation*}
h(u):=\sum_{k=1}^{|u|} \ell(u_k).
\end{equation*}
We denote by $H^1$ the height process of $\F^1$, and we define $H^{\ell}$ the weighted depth-first exploration process of $\F$ as follows~:
\begin{equation}\label{eq:defH}
\forall n\in\N,\;H^1(n):=|u^1(n)|\;\textrm{ and }\;H^{\ell}(n):=h(u(n)). 
\end{equation}
Notice that one of the differences between $H^1$ and $H^{\ell}$ is that in $H^1$, $\ell$ has no influence, whereas in $H^{\ell}$ it does. As explained in \cite{duquesne-le-gall}, these processes characterise $\F^1$ and $(\F,\ell)$ (information on $\ell$ is easily recovered from $H^\ell$). Let us state our first main result~(i), together with corollary results~(ii) and~(iii)~: 
\begin{theorem}{\label{th:bitype}}
Let $(\F,e,h)$ be a leafed Galton--Watson forest with edge lengths, with offspring distribution $\zeta$ satisfying hypothesis $\bf (H)$. 
\begin{itemize}
\item[\rm (i)] The following convergence in law holds for the Skorokhod topology on the space $\mathbb{D}(\R_+,\R)$ of c\`adl\`ag functions~: 
\begin{equation*}
\left(\frac{H^{\ell}(\fl{ns})}{\sqrt{n}},\frac{H^1(\fl{ns})}{\sqrt{n}}\right)_{s≥ 0} \substack { \Longrightarrow \\ n\to\infty} \left(\frac{2\mu}{\sigma}|B_{m^{-1}s}|,\frac{2}{\sigma}|B_{s}|\right)_{s≥ 0},
\end{equation*}
where $B$ is a standard Brownian motion. 
\item[\rm (ii)] For all $n\in\N$, let $\Gamma_n$ be the index of the tree to which $u(n)$ belongs. Then the following convergence in law holds jointly with that of~{\rm (i)}~: 
\begin{equation*}
\left( \frac{\Gamma_{\fl{ns}}}{\sqrt{n}} \right)_{s≥0} \sto[\Longrightarrow]{n\to\infty} \left( \sigma L_{m^{-1}s}^0 \right)_{s≥0}, 
\end{equation*}
where $(L_s^0)_{s≥0}$ is the local time at level $0$ of $B$, the Brownian motion of {\rm (i)}, normalised as the occupation density of $B$ at $0$. 
\item[\rm (iii)] Let $h_{max}(\T)=\max_{u\in\T}h(u)$ be the weighted height of the tree $\T$. Then, 
\begin{equation*}
n\P\Big( h_{max}(\T)≥n \Big) \sto[\Longrightarrow]{n\to\infty} \frac{2\mu}{\sigma^2}. 
\end{equation*}
\end{itemize}
\end{theorem}
\noindent Notice that (i) implies the convergence in law of $\F$ and $\F^1$ properly rescaled towards the same Brownian forest for the Gromov-Hausdorff topology (see Lemma 2.4 of~\cite{le-gall-rt}). The convergence of the marginal distribution of the second component in~(i) is Theorems~2.3.1 and~2.3.2 of~\cite{duquesne-le-gall}. The proof of Theorem~\ref{th:bitype} will be carried out in Section~\ref{s:proof1}. The method used to prove Theorem~\ref{th:bitype}~(i) is inspired by the proof of Theorem~1~(i) in~\cite{miermont} ; we will show that we can get $H^1$ close to $H^{\ell}$ for the Skorokhod topology on càdlàg functions, after an adequate scaling on two directions~: 
\begin{itemize}
\item[$\bullet$] on the amplitude of $H^1$, by a factor $\mu$~; we will show in Proposition~\ref{prop:vertical} that this is what it takes to get $H^1$ "vertically close" to $H^{\ell}$. 
\item[$\bullet$] "in time", by a factor $m^{-1}$, in order to "slow down" the depth-first exploration process on $\F^1$ for it to follow that on $\F$. Indeed, unlike $H^{\ell}$, $H^1$ does not visit vertices of type $0$, which makes it go faster. We will show in Proposition~\ref{prop:horizontal} that $m^{-1}$ is the right pace. 
\end{itemize}
As for the proof of Theorem~\ref{th:bitype}~(ii) and~(iii), it will follow that of Theorem~1~(ii) and Corollary~1 of \cite{miermont}, and it will be outlined in Section~\ref{s:conclusion1}. 
We emphasise that G.~Miermont's theorem in~\cite{miermont} cannot be simply applied here to get rid of the 2-type constraint. Indeed, in our case, it is not possible for a vertex of type 0 to have a descendant of type $1$ (in the setting of \cite{miermont}, the \textit{mean matrix} is not \textit{irreducible}).

\subsection{Multitype Galton--Watson trees}\label{s:intromulti}

Let us introduce a more classic process, the multitype Galton--Watson tree. Multitype Galton--Watson trees are trees to each vertex of which a type is associated. They are built in a way such that the progeny of each vertex is independent of that of other vertices, but such that the law of the progeny depends on the type of the vertex. Usually, they are studied in the case where the set of possible types is finite, mainly because of the importance of the \textit{mean matrix}, which has to be of finite-dimension if one wants to apply the Perron-Frobenius theorem to it, and thus characterise the behaviour of the tree. This case is well discussed by T.~E.~Harris in Chapter~II of \cite{harris}. However, one may consider more general sets of types. In what follows, we will suppose that the set of types is countable, and we will see that under good conditions on the mean matrix, it is possible to obtain the same tools than in the case where this set is finite (we will strongly rely on Chapter 6 of~\cite{seneta} for this). \\

Let $\X$ be a countable set (to which we will refer as the set of types), and ${\bzeta}=(\zeta_x)_{x\in\X}$ a family of laws taking their values in $\X^{(\N)}$ (the set of finite sequences of $\X$, including the empty sequence). A realisation of a random variable $Z$ of law $\zeta_x$, where $x \in \X$, gives the make-up of the progeny of a vertex of type $x$ in this way~:
\begin{itemize}
\item[$\bullet$] The length of $Z$ (denoted by $|Z|$) is the number of children of the vertex. 
\item[$\bullet$] The list of types forming $Z$ gives the list of the types of each of the $|Z|$ children~; if the latter is $(1,2,1)$ for example, then it means that the first child of the vertex is of type $1$, the second is of type $2$ and the third is of type $1$. 
\end{itemize}
\medskip

\noindent Let $x_0\in\X$. We consider in this part $(\Tb,\ty)$ a multitype Galton--Watson tree with offspring distribution $\bzeta$ (for any $u\in\Tb$, $\ty(u)\in\X$ denotes the type of $u$) and initial type $x_0$, that is $(\Tb,\ty)$ is built by induction on generations as follows~: 
\begin{itemize}
\item {\bf Initialisation} \\
Generation $0$ of $\Tb$ is only made up of the root, denoted by $\root$, with type $\ty(\root)=x_0$. 
\item {\bf Induction} \\
Let $n≥0$, and suppose that $\Tb$ has been built up to generation $n$. If generation $n$ is empty, then generation $n+1$ is empty. Otherwise, each vertex $u$ of generation $n$ gives progeny according to $\zeta_{\ty(u)}$, independently of other vertices, thus forming generation $n+1$. 
\end{itemize}

For $x_0\in\X$, we denote by $\P_{x_0}$ the probability law of $\Tb$ under which its root has type $x_0$, and $\E_{x_0}$ the associated expectation. We also let $\Fb$ be a multitype Galton--Watson forest with offspring distribution $\bzeta$, that is a collection of \iid  multitype Galton--Watson trees with offspring distribution $\bzeta$. For any $x_0\in\X$, we let $\P_{x_0}$ be the probability under which all the trees composing $\Fb$ have a root of type $x_0$, and $\E_{x_0}$ the associated expectation. 

We will use the general notation introduced in Subsection~\ref{s:introleafed} ; moreover we let for all $y\in\X$ and for all $u\in\Tb$, $\nu^{y}(u)$ be the number of children of type $y$ of $u$. \\

Let us introduce some conditions on our process. Let ${\bf M}=(m_{x,y})_{x,y\in\X}$ be the {\it mean matrix} of our process, where for all $x,y\in \X$,
\begin{equation*}
m_{x,y}:= \E_{x}\Big[\nu^{y}\Big], 
\end{equation*}
that is $m_{x,y}$ is the mean number of children of type $y$ of a vertex of type $x$. In Chapter~III of~\cite{harris}, the author studies multitype Galton--Watson processes with general sets of types under a condition of uniformity on the coefficients of $\bf M$. Our study will require weaker hypotheses on this matrix. First, we will suppose that all iterate coefficients of $\bf M$ are finite, that is 
\begin{equation*}
 \forall x,y\in\X,\: \forall k\in\Ns,\: m_{x,y}^{(k)}<\infty
\end{equation*} 
where $m_{x,y}^{(k)}$ is defined by induction as follows :
\begin{equation*}
m_{x,y}^{(1)}:=m_{x,y},\:m_{x,y}^{(k+1)}:=\sum_{z\in\X}m_{x,z}^{(k)} m_{z,y}\textrm{ for }k\geq 1.
\end{equation*} 
We also suppose that $\bf M$ is irreducible, that is such that for all $x,y\in\X$ there exists $k≥1$ such that $ m_{x,y}^{(k)}>0$. \\

In the case where $\X$ is finite, the Perron-Frobenius theorem can be applied to $\bf M$ : it ensures the existence of a maximal eigenvalue that is simple, to which are associated a right and a left eigenvector with positive entries (the only ones up to a constant to have positive entries). A necessary condition for Theorem~\ref{th:multitype} to hold in this case is this eigenvalue to be equal to $1$, a condition thus equivalent to the existence of left and right eigenvectors associated to $1$ with positive entries. 

In the general case where $\X$ is countable but not finite, the Perron-Frobenius theorem cannot be applied. However, denoting by $R$ the common convergence radius of ${\bf M}$ (also called the convergence parameter) defined in Chapter 6.1 in~\cite{seneta} (p.\ 200), according to Theorem 6.2 in~\cite{seneta}, if $R=1$, then there exist positive left and right eigenvectors of ${\bf M}$.  On the other hand, suppose $\bf (H_M)$ below, then Theorem 6.4 in~\cite{seneta} guarantees that $R=1$ (a condition that matches that of the finite case), and that $\bf M$ is $1$-positive (in the sense of Definition 6.2). It is therefore quite natural to make the following hypotheses on ${\bf M}$~: \\ 
\begin{equation*}
{\bf (H_M)}\begin{cases}\parbox{0.9\textwidth}{$\bf M$ is irreducible with finite iterate coefficients, and there exist $(a_x)_{x\in\X}\in{\R_+^*}^{\X}$ a left eigenvector of $\bf M$ associated to eigenvalue $1$ and $(b_x)_{x\in\X}\in{\R_+^*}^{\X}$ a right eigenvector of $\bf M$ associated to eigenvalue $1$, such that 
\begin{equation*}
\sum_{x\in\X} a_x<\infty,\; \sum_{x\in\X} a_x b_x<\infty, 
\end{equation*}
and renormalised so that $\sum_{x\in\X} a_x=1$ and $\sum_{x\in\X} a_x b_x=1$. }\end{cases}
\end{equation*}
Note that if these two vectors exist, then they are unique up to a constant (Theorem 6.4 of~\cite{seneta}). Notice also that the only extra condition compared to the case where $\X$ is finite is the finiteness of the two sums (a condition always satisfied in this case). \\

We also need hypotheses on second order moments. Let us set for all $x,y,z\in\X$,
\begin{equation*}
Q^x_{y,z} := \E_x\Big[ \Big(\sum_{|u|=1}\1{ \ty(u)=y }\Big)\Big(\sum_{|u|=1}\1{ \ty(u)=z }\Big) \Big]-\delta_{y,z}m_{x,z}, 
\end{equation*}
and let us make the following hypotheses, which also appears in \cite{miermont} : \\
\begin{equation*}
{\bf (H_Q)}\begin{cases}\parbox{\textwidth}{
\begin{itemize}
\item[$\bullet$] For all $x,y,z\in\X$, $Q_{y,z}^x<\infty$, 
\item[$\bullet$] $\eta:=\sqrt{\sum_{x,y,z\in\X}a_x b_y Q^x_{y,z} b_z}<\infty$. 
\end{itemize}}\end{cases}
\end{equation*}
This constant squared, $\eta^2$, will turn out to be the equivalent of the variance in the monotype case~; hypothesis $\bf (H_Q)$ is therefore necessary to our theorem. \\

Finally, we introduce a last hypothesis. Let $x\in\X$, we set
\begin{equation*}
\bf (H_R^{x})\begin{cases}
\parbox{\textwidth}{\begin{itemize} 
\item[$\bullet$] $y^2\P_x\Big( \max\{ |u| \st u\in\Tb,\: \ty(u_1),\ldots,\ty(u_{|u|})\neq x \}>y \Big)\sto{y\to\infty} 0$, 
\item[$\bullet$] $y^2\E_x\Big( \sum_{|u|>y}\1{ \ty(u_1),\ldots,\ty(u_{|u|-1})\neq x,\:\ty(u)=x} \Big)\sto{y\to\infty} 0$. 
\end{itemize}}
\end{cases}
\end{equation*}

In the appendix, we will give a stronger but simpler hypothesis implying $\bf (H_R^{x})$ for any $x$, which will always be satisfied in the case where $\X$ is finite. Let us now state our theorem. 

\begin{theorem}\label{th:multitype}
Let $x_0\in\X$, and let $(\Fb,\ty)$ be a multitype Galton--Watson forest such that hypotheses $\bf (H_M)$, $\bf (H_Q)$ and $\bf (H_R^{x_0})$ are satisfied. 
\begin{itemize}
\item[\rm (i)] Under $\P_{x_0}$, the following convergence in law holds for the Skorokhod topology on the space $\mathbb{D}(\R_+,\R)$ of c\`adl\`ag functions~: 
\begin{equation*}
\Big(\frac{|u(\fl{ns})|}{\sqrt{n}}\Big)_{s≥ 0} \substack { \Longrightarrow \\ n\to\infty} \Big(\frac{2}{\eta}|B_s|\Big)_{s≥ 0},
\end{equation*}
where $B$ is a standard Brownian motion. 
\item[\rm (ii)] For all $n\in\N$, let $\Gamma_n$ be the index of the tree to which $u(n)$ belongs. Then, under $\P_{x_0}$, the following convergence in law holds jointly with that of {\rm (i)}~: 
\begin{equation*}
\left( \frac{\Gamma_{\fl{ns}}}{\sqrt{n}} \right)_{s≥0} \sto{n\to\infty} \left( \frac{\sigma}{b_{x_0}} L_s^0 \right)_{s≥0}, 
\end{equation*}
where $(L_s^0)_{s≥0}$ is the local time of $B$, the Brownian motion of {\rm (i)}, normalised as the occupation density of $B$ at $0$.  
\item[\rm (iii)] Let $h_{max}(\Tb)=\max_{u\in\Tb}|u|$ be the height of the tree $\Tb$. Then, 
\begin{equation*}
n\P_{x_0}\Big( h_{max}(\T)≥n \Big) \sto{n\to\infty} \frac{2b_{x_0}}{\eta^2}. 
\end{equation*}
\end{itemize}
\end{theorem}

\noindent This theorem was proved by G.~Miermont in \cite{miermont} in the case where $\X$ is finite, with optimal hypotheses under the assumption of irreducibility of the mean matrix $\bf M$~: Perron-Frobenius eigenvalue~$1$, and condition $\bf (H_Q)$. Our conditions may seem more restrictive~; however, in the finite case, they are implied by these optimal conditions. Indeed, in that case, hypothesis $\bf (H_R^{x})$ is always satisfied for any $x$ (see the appendix), and supposing the irreducibility of the mean matrix with Perron-Frobenius eigenvalue $1$ would imply $\bf (H_M)$, according to Theorem~6.2 of~\cite{seneta}. 

Theorem~\ref{th:multitype} therefore extends G.~Miermont's~one to the case where the set $\X$ is countable. Our proof will be different from his, although inspired by it. Indeed, in the latter, the author used an inductive method on the total number of types : given a multitype Galton--Watson tree with say $K$ types (where $K≥1)$, he was able to build a multitype Galton--Watson tree with $K-1$ types the height function of which was close (up to a re-normalisation) to that of the first tree. Then, step by step, he was able to show that the height function of the original tree was close to that of a monotype Galton--Watson tree, to which T. Duquesne and J.-F. Le Gall's~theorem \cite{duquesne-le-gall} could be applied. Obviously, this method cannot be used in our case. 

So we will introduce in the next subsection (Subsection~\ref{s:reduction}) a reduction (inspired by that of \cite{miermont}) which associates to any multitype Galton--Watson tree a leafed Galton--Watson tree with edge lengths. The whole point of this reduction is that it is such that both trees have the same height function (this is Proposition~\ref{prop:egaliteprocess}). Then, we will just have to show that if the multitype tree satisfies the hypotheses introduced earlier, then the tree obtained by this reduction satisfies hypothesis $\bf (H)$. This will be done using change of measure techniques in Section~\ref{s:proof2}, and thus according to Theorem~\ref{th:bitype} its height function will converge towards the reflected Brownian motion, and so will that of the multitype tree. \\

\subsection{Reduction of multitype trees to leafed trees with edge lengths}\label{s:reduction}

Let us introduce the method of construction of $(\T,e,\ell)$ a leafed Galton--Watson tree with edge lengths the associated depth-first exploration process of which is equal to the height process of a given multitype Galton--Watson tree $(\Tb,\ty)$. To this end, let us define the notion of \textit{optional line of a given type}. 
\begin{definition}
Let $y\in\X$ and $u\in\Tb$. 
\begin{itemize}
\item We denote by ${\B}_u^y$ the set of vertices descending from $u$ in $\Tb$ having no ancestor of type $y$ since $u$. Formally, 
\begin{equation*}
{\B}_u^y=\{ v\in\Tb \st u\vdash v\text{ and } \ty(w)\neq y \quad \forall w\in\Tb \text{ such that } u\vdash w \vdash v \}. 
\end{equation*}
\item We denote by ${\cal{L}}_u^y$ the set of vertices of type $y$ descending from $u$ in $\Tb$ and having no ancestor of type $y$ since $u$. Formally, 
\begin{equation*}
{\cal{L}}_u^y=\{ v\in\Tb \st u\vdash v,\: \ty(v)=y,\: \ty(w)\neq y \quad \forall w\in\Tb \text{ such that } u\vdash w \vdash v \}. 
\end{equation*} 
\end{itemize}
When $u=\root$, we will denote ${\cal{L}}_u^y$ by ${\cal{L}}^y$ and ${\B}_u^y$ by ${\B}^y$. 
\end{definition}

\noindent We say that ${\cal{L}}_u^y$ is the optional line of type $y$ stemming from $u$.  Somehow, ${\cal{L}}_u^y$ is the "top layer" of ${\B}_u^y$. The basic framework of optional lines was established in~\cite{jagers}. Of course, this notation can be extended to forests. Under $\P_{x_0}$, the construction of $(\T,e,\ell)$ consists in adding a vertex of type $1$ in $\T$ for each vertex of type $x_0$ in $\Tb$, and a vertex of type $0$ in $\T$ for each vertex of type $\neq x_0$ in $\Tb$. It is carried out inductively as follows~:  
\begin{itemize}
\item {\bf Initialisation} \\
Generation $0$ of $\T$ is made up of a root, $\root$, and we set $e(\root)=1$. Let us build generation~$1$. Take, in the lexicographical order, the vertices $v\in\Tb$ such that $v\in{\B}_{\root}^{x_0}$. Following their lexicographical ordering, to each $v\in\Tb$ among these vertices we associate a vertex $v^{x_0}$ to the first generation of $\T$, setting $e(v^{x_0})=1$ if $\ty(v)=x_0$ (that is if $v\in\cal{L}^{x_0}$), $e(v^{x_0})=0$ otherwise. Moreover, for each of these vertices $v^{x_0}\in\T$, we set its edge length as $\ell(v^{x_0})=|v|$. 
\item {\bf Induction} \\
Let $n≥1$, and suppose that generation $n$ of $\T$ has been built. If generation $n$ of $\T$ is empty then generation $n+1$ of $\T$ is empty. Otherwise, for each $u^{x_0}\in\T$ of the $n$\up{th} generation of $\T$ such that $e(u^{x_0})=1$, take in the lexicographical order the vertices $v\in\Tb$ such that $v\in{\B}_u^{x_0}$. Proceeding in the lexicographical order, to each $v\in\Tb$ of these vertices, we associate a vertex $v^{x_0}$ as a child of $u^{x_0}$ in $\T$, thus forming the progeny of $u^{x_0}$. We set $e(v^{x_0})=1$ if $\ty(u)=x_0$ (that is if $v\in\cal{L}_u^{x_0}$) and $e(v^y)=0$ otherwise. Then, for each of these vertices $v^{x_0}\in\T$, we set $\ell(v^{x_0})=|v|-|u|$. 
\end{itemize}  
\medskip

\begin{figure}[H]
\begin{tikzpicture}[line cap=round,line join=round,>=triangle 45,x=0.9cm,y=0.9cm]
\clip(4.2,-2.5) rectangle (22.8,7.5);
\draw[|->] (4.4,0)-- (4.4,6.7) node[above]{\hspace{1.1cm}\small $|u_{\Tb}(n)|$\phantom{))}};
\foreach \y in {1.0,2.0,3.0,4.0,5.0,6.0}
\draw[shift={(4.4,\y)},color=black] (2pt,0pt) -- (-2pt,0pt);
\draw[|->] (22.4,0)-- (22.4,6.7) node[above]{\small $H^\ell(n)$\phantom{))}};
\foreach \y in {1.0,2.0,3.0,4.0,5.0,6.0}
\draw[shift={(22.4,\y)},color=black] (2pt,0pt) -- (-2pt,0pt);
\draw (8,7.5) node[anchor=north west] {\large $\mathbf{T}$};
\draw (17,7.5) node[anchor=north west] {\large $\mathbb{T}$};
\draw [color=blue] (5,-1) ++(-2.0pt,0 pt) -- ++(2.0pt,2.0pt)--++(2.0pt,-2.0pt)--++(-2.0pt,-2.0pt)--++(-2.0pt,2.0pt);
\draw [fill=red] (5,-2) ++(-3.0pt,0 pt) -- ++(3.0pt,3.0pt)--++(3.0pt,-3.0pt)--++(-3.0pt,-3.0pt)--++(-3.0pt,3.0pt);
\draw (5,-0.5) node[anchor=north west] {Vertices of type $\neq x_0$};
\draw (5,-1.5) node[anchor=north west] {Vertices of type $x_0$};
\draw [dash pattern=on 3pt off 3pt,color=gray] (11,-0.5) -- (11,-1) -- (10.5,-1) -- (10.5,-0.5) -- (11,-0.5);
\draw (11,-0.5) node[anchor=north west] {The set $\B^{x_0}_{\root}$};
\draw [dotted, color=gray] (11,-1.5)-- (10.5,-1.5)-- (10.5,-2)-- (11,-2)-- (11,-1.5);
\draw (11,-1.5) node[anchor=north west] {First generation of $\T$};
\draw [color=blue] (17,-1) circle (2.0pt);
\draw [fill=red] (17,-2) circle (2.5pt);
\draw (17,-0.5) node[anchor=north west] {Vertices of type $0$};
\draw (17,-1.5) node[anchor=north west] {Vertices of type $1$};
\draw [->] (12.7,3.5)-- (13.7,3.5);

\draw[dash pattern=on 2pt off 2pt,color=gray] (5.5,0.5) -- (5.5,2.5) -- (6.5,2.5) -- (6.5,2.0) -- (7.5,2.0) -- (7.5,3.5) -- (8.5,3.5) -- (8.5,1.5) -- (9.5,1.5) -- (9.5,1.0) -- (10.5,1.0) -- (10.5,1.5) -- (11.5,1.5) -- (11.5,0.5) -- (8.5,0.5) -- (7.5,1.5) -- cycle;
\draw[dotted,color=gray] (15.0,0.5) -- (15.0,2.5) -- (16.0,2.5) -- (16.0,2.0) -- (17.0,2.0) -- (17.0,3.5) -- (18.0,3.5) -- (18.0,1.5) -- (19.0,1.5) -- (19.0,1.0) -- (20.0,1.0) -- (20.0,1.5) -- (21.0,1.5) -- (21.0,0.5) -- (18.0,0.5) -- (17.0,1.5) -- cycle;
\draw (9.0,0.0)-- (6.0,1.0);
\draw (9.0,0.0)-- (9.0,1.0);
\draw (9.0,0.0)-- (11.0,1.0);
\draw (6.0,1.0)-- (6.0,2.0);
\draw (6.0,1.0)-- (8.0,2.0);
\draw (11.0,1.0)-- (10.0,2.0);
\draw (11.0,1.0)-- (11.0,2.0);
\draw (6.0,2.0)-- (5.0,3.0);
\draw (6.0,2.0)-- (5.5,3.0);
\draw (6.0,2.0)-- (7.0,3.0);
\draw (8.0,2.0)-- (8.0,3.0);
\draw (10.0,2.0)-- (10.0,3.0);
\draw (10.0,2.0)-- (10.5,3.0);
\draw (11.0,2.0)-- (12.0,3.0);
\draw (10.0,3.0)-- (9.0,4.0);
\draw (10.0,3.0)-- (10.0,4.0);
\draw (12.0,3.0)-- (11.0,4.0);
\draw (12.0,3.0)-- (12.0,4.0);
\draw (5.5,3.0)-- (5.5,4.0);
\draw (5.5,3.0)-- (6.5,4.0);
\draw (10.0,4.0)-- (9.5,5.0);
\draw (10.0,4.0)-- (10.5,5.0);
\draw (5.5,4.0)-- (5.0,5.0);
\draw (5.0,5.0)-- (5.0,6.0);
\draw (5.0,5.0)-- (6.0,6.0);
\draw (9.5,5.0)-- (9.5,6.0);
\draw (15.5,2.0)-- (14.5,3.0);
\draw (15.5,2.0)-- (15.0,3.0);
\draw (15.5,2.0)-- (16.5,3.0);
\draw (19.5,3.0)-- (18.5,4.0);
\draw (19.5,3.0)-- (19.5,4.0);
\draw (21.5,3.0)-- (20.5,4.0);
\draw (21.5,3.0)-- (21.5,4.0);
\draw (19.5,4.0)-- (19.0,5.0);
\draw (19.5,4.0)-- (20.0,5.0);
\draw (15.0,4.0)-- (14.5,5.0);
\draw (18.5,0.0)-- (15.5,1.0);
\draw (18.5,0.0)-- (18.5,1.0);
\draw (18.5,0.0)-- (20.5,1.0);
\draw (18.5,0.0)-- (15.5,2.0);
\draw (18.5,0.0)-- (17.5,2.0);
\draw (18.5,0.0)-- (17.5,3.0);
\draw (20.5,1.0)-- (19.5,2.0);
\draw (20.5,1.0)-- (19.5,3.0);
\draw (20.5,1.0)-- (20.0,3.0);
\draw (20.5,1.0)-- (20.5,2.0);
\draw (20.5,1.0)-- (21.5,3.0);
\draw (19.5,4.0)-- (19.0,6.0);
\draw (15.0,4.0)-- (14.5,6.0);
\draw (15.0,4.0)-- (15.5,6.0);
\draw (15.5,2.0)-- (15.0,4.0);
\draw (15.5,2.0)-- (16.0,4.0);
\begin{scriptsize}
\draw [fill=red] (9.0,0.0) ++(-3.0pt,0 pt) -- ++(3.0pt,3.0pt)--++(3.0pt,-3.0pt)--++(-3.0pt,-3.0pt)--++(-3.0pt,3.0pt);
\draw [color=blue] (6.0,1.0) ++(-2.0pt,0 pt) -- ++(2.0pt,2.0pt)--++(2.0pt,-2.0pt)--++(-2.0pt,-2.0pt)--++(-2.0pt,2.0pt);
\draw [color=blue] (9.0,1.0) ++(-2.0pt,0 pt) -- ++(2.0pt,2.0pt)--++(2.0pt,-2.0pt)--++(-2.0pt,-2.0pt)--++(-2.0pt,2.0pt);
\draw [fill=red] (11.0,1.0) ++(-3.0pt,0 pt) -- ++(3.0pt,3.0pt)--++(3.0pt,-3.0pt)--++(-3.0pt,-3.0pt)--++(-3.0pt,3.0pt);
\draw [fill=red] (6.0,2.0) ++(-3.0pt,0 pt) -- ++(3.0pt,3.0pt)--++(3.0pt,-3.0pt)--++(-3.0pt,-3.0pt)--++(-3.0pt,3.0pt);
\draw [color=blue] (8.0,2.0) ++(-2.0pt,0 pt) -- ++(2.0pt,2.0pt)--++(2.0pt,-2.0pt)--++(-2.0pt,-2.0pt)--++(-2.0pt,2.0pt);
\draw [color=blue] (10.0,2.0) ++(-2.0pt,0 pt) -- ++(2.0pt,2.0pt)--++(2.0pt,-2.0pt)--++(-2.0pt,-2.0pt)--++(-2.0pt,2.0pt);
\draw [color=blue] (11.0,2.0) ++(-2.0pt,0 pt) -- ++(2.0pt,2.0pt)--++(2.0pt,-2.0pt)--++(-2.0pt,-2.0pt)--++(-2.0pt,2.0pt);
\draw [color=blue] (5.0,3.0) ++(-2.0pt,0 pt) -- ++(2.0pt,2.0pt)--++(2.0pt,-2.0pt)--++(-2.0pt,-2.0pt)--++(-2.0pt,2.0pt);
\draw [color=blue] (5.5,3.0) ++(-2.0pt,0 pt) -- ++(2.0pt,2.0pt)--++(2.0pt,-2.0pt)--++(-2.0pt,-2.0pt)--++(-2.0pt,2.0pt);
\draw [color=blue] (7.0,3.0) ++(-2.0pt,0 pt) -- ++(2.0pt,2.0pt)--++(2.0pt,-2.0pt)--++(-2.0pt,-2.0pt)--++(-2.0pt,2.0pt);
\draw [fill=red] (8.0,3.0) ++(-3.0pt,0 pt) -- ++(3.0pt,3.0pt)--++(3.0pt,-3.0pt)--++(-3.0pt,-3.0pt)--++(-3.0pt,3.0pt);
\draw [fill=red] (10.0,3.0) ++(-3.0pt,0 pt) -- ++(3.0pt,3.0pt)--++(3.0pt,-3.0pt)--++(-3.0pt,-3.0pt)--++(-3.0pt,3.0pt);
\draw [color=blue] (10.5,3.0) ++(-2.0pt,0 pt) -- ++(2.0pt,2.0pt)--++(2.0pt,-2.0pt)--++(-2.0pt,-2.0pt)--++(-2.0pt,2.0pt);
\draw [fill=red] (12.0,3.0) ++(-3.0pt,0 pt) -- ++(3.0pt,3.0pt)--++(3.0pt,-3.0pt)--++(-3.0pt,-3.0pt)--++(-3.0pt,3.0pt);
\draw [color=blue] (9.0,4.0) ++(-2.0pt,0 pt) -- ++(2.0pt,2.0pt)--++(2.0pt,-2.0pt)--++(-2.0pt,-2.0pt)--++(-2.0pt,2.0pt);
\draw [fill=red] (10.0,4.0) ++(-3.0pt,0 pt) -- ++(3.0pt,3.0pt)--++(3.0pt,-3.0pt)--++(-3.0pt,-3.0pt)--++(-3.0pt,3.0pt);
\draw [color=blue] (11.0,4.0) ++(-2.0pt,0 pt) -- ++(2.0pt,2.0pt)--++(2.0pt,-2.0pt)--++(-2.0pt,-2.0pt)--++(-2.0pt,2.0pt);
\draw [color=blue] (12.0,4.0) ++(-2.0pt,0 pt) -- ++(2.0pt,2.0pt)--++(2.0pt,-2.0pt)--++(-2.0pt,-2.0pt)--++(-2.0pt,2.0pt);
\draw [fill=red] (5.5,4.0) ++(-3.0pt,0 pt) -- ++(3.0pt,3.0pt)--++(3.0pt,-3.0pt)--++(-3.0pt,-3.0pt)--++(-3.0pt,3.0pt);
\draw [color=blue] (6.5,4.0) ++(-2.0pt,0 pt) -- ++(2.0pt,2.0pt)--++(2.0pt,-2.0pt)--++(-2.0pt,-2.0pt)--++(-2.0pt,2.0pt);
\draw [color=blue] (9.5,5.0) ++(-2.0pt,0 pt) -- ++(2.0pt,2.0pt)--++(2.0pt,-2.0pt)--++(-2.0pt,-2.0pt)--++(-2.0pt,2.0pt);
\draw [color=blue] (10.5,5.0) ++(-2.0pt,0 pt) -- ++(2.0pt,2.0pt)--++(2.0pt,-2.0pt)--++(-2.0pt,-2.0pt)--++(-2.0pt,2.0pt);
\draw [color=blue] (5.0,5.0) ++(-2.0pt,0 pt) -- ++(2.0pt,2.0pt)--++(2.0pt,-2.0pt)--++(-2.0pt,-2.0pt)--++(-2.0pt,2.0pt);
\draw [color=blue] (5.0,6.0) ++(-2.0pt,0 pt) -- ++(2.0pt,2.0pt)--++(2.0pt,-2.0pt)--++(-2.0pt,-2.0pt)--++(-2.0pt,2.0pt);
\draw [color=blue] (6.0,6.0) ++(-2.0pt,0 pt) -- ++(2.0pt,2.0pt)--++(2.0pt,-2.0pt)--++(-2.0pt,-2.0pt)--++(-2.0pt,2.0pt);
\draw [fill=red] (9.5,6.0) ++(-3.0pt,0 pt) -- ++(3.0pt,3.0pt)--++(3.0pt,-3.0pt)--++(-3.0pt,-3.0pt)--++(-3.0pt,3.0pt);
\draw [color=blue] (15.5,1.0) circle (2.0pt);
\draw [color=blue] (18.5,1.0) circle (2.0pt);
\draw [fill=red] (20.5,1.0) circle (2.5pt);
\draw [fill=red] (15.5,2.0) circle (2.5pt);
\draw [color=blue] (17.5,2.0) circle (2.0pt);
\draw [color=blue] (20.5,2.0) circle (2.0pt);
\draw [color=blue] (14.5,3.0) circle (2.0pt);
\draw [color=blue] (15.0,3.0) circle (2.0pt);
\draw [color=blue] (16.5,3.0) circle (2.0pt);
\draw [fill=red] (17.5,3.0) circle (2.5pt);
\draw [fill=red] (19.5,3.0) circle (2.5pt);
\draw [color=blue] (20.0,3.0) circle (2.0pt);
\draw [fill=red] (21.5,3.0) circle (2.5pt);
\draw [color=blue] (18.5,4.0) circle (2.0pt);
\draw [fill=red] (19.5,4.0) circle (2.5pt);
\draw [color=blue] (20.5,4.0) circle (2.0pt);
\draw [color=blue] (21.5,4.0) circle (2.0pt);
\draw [fill=red] (15.0,4.0) circle (2.5pt);
\draw [color=blue] (16.0,4.0) circle (2.0pt);
\draw [color=blue] (19.0,5.0) circle (2.0pt);
\draw [color=blue] (20.0,5.0) circle (2.0pt);
\draw [color=blue] (14.5,5.0) circle (2.0pt);
\draw [color=blue] (14.5,6.0) circle (2.0pt);
\draw [color=blue] (15.5,6.0) circle (2.0pt);
\draw [fill=red] (19.0,6.0) circle (2.5pt);
\draw [fill=red] (18.5,0.0) circle (2.5pt);
\draw [color=blue] (19.5,2.0) circle (2.0pt);
\end{scriptsize}
\end{tikzpicture}
\caption{A realisation of $\Tb$ under $\P_{x_0}$, and the tree $\T$ resulting from it. }
\label{f:transformation}
\end{figure}
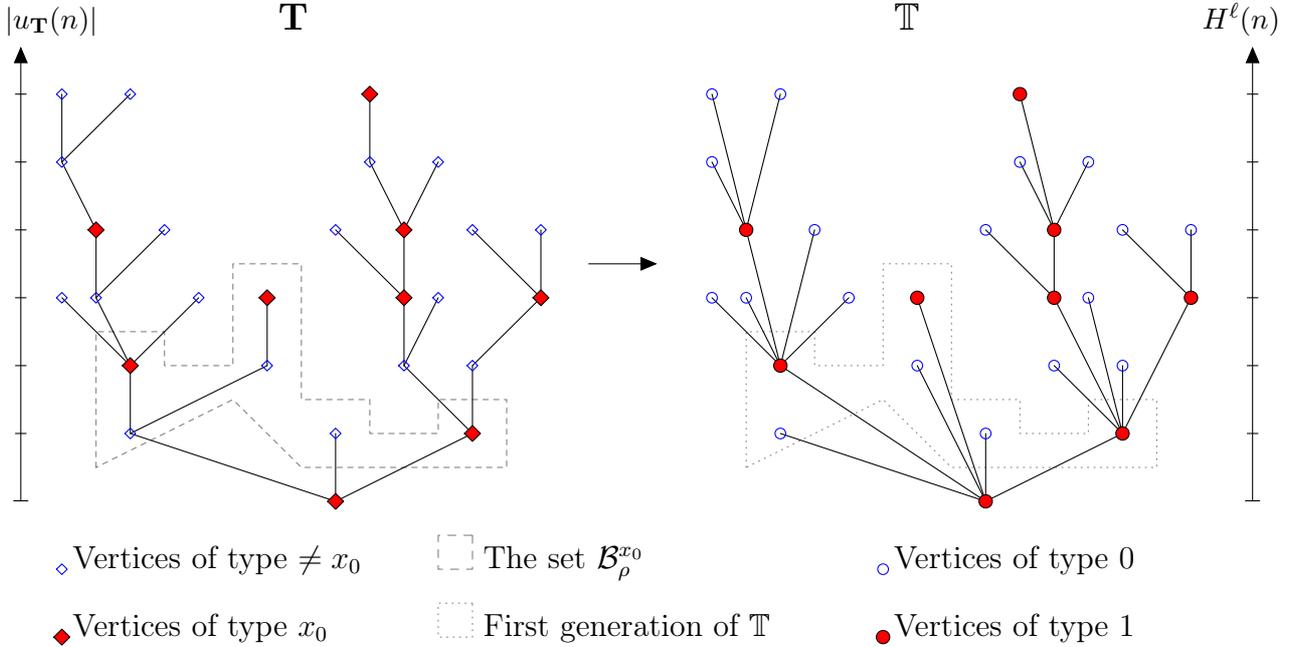
Constructing $\T$ from $\Tb$ therefore consists in untangling the "bushes" $\B_u^{x_0}$ stemming from vertices $u\in\Tb$ such that $\ty(u)=x_0$, so that all vertices are in the same generation, however keeping their lexicographical ordering, and keeping in $\ell$ the information on their initial generation in the tree $\Tb$. Of course, the construction of $\T$ from $\Tb$ can be extended to forests $\Fb$ as $\F$, by applying the reduction to each component tree. The whole purpose of this construction lies in the following proposition. 
\begin{proposition}\label{prop:egaliteprocess}
Under $\P_{x_0}$, the marked tree $(\T,e,\ell)$ is a leafed Galton--Watson tree with edge lengths. Moreover, denoting by $H^{\ell}$ the depth-first exploration process introduced in~(\ref{eq:defH}) associated to it, we have
\begin{equation*}
\forall n\in\N,\: H^{\ell}(n)=|u_{\Tb}(n)|,
\end{equation*}
where $u_{\Tb}(n)$ is the $n$\up{th} vertex of $\Tb$ for the lexicographic order. 
\end{proposition}
\begin{proof}
The branching property in $\Tb$ guarantees that the progenies of each vertex in $\T$ have same law, and then by construction $(\T,e,\ell)$ is a leafed Galton--Watson tree with edge lengths. The equality of depth-first exploration processes also naturally stems from the construction. 
\end{proof}
Therefore, in order to show Theorem~\ref{th:multitype}, we will just have to prove that $\F$ satisfies the hypotheses introduced in Subsection~\ref{s:introleafed} and to apply Theorem~\ref{th:bitype} to it~; that is we need to prove~:
\begin{proposition}\label{prop:hypotheses}
Under conditions $\bf (H_M)$, $\bf (H_Q)$ and $\bf (H_R^{x_0})$ on $\Fb$, $(\F,e,\ell)$ satisfies hypothesis $\bf (H)$ introduced in Part I. 
\end{proposition}
It is quite straightforward that under $\P_{x_0}$, if $\Fb$ satisfies condition $\bf (H_R^{x_0})$, then conditions denoted by $\bf (H_0^2)$ and $\bf (H_1^2)$ in $\bf (H)$ are satisfied by $\F$. In Subsections~\ref{s:proofh1} and~\ref{s:proofhc}, we will prove that hypotheses $\bf (H_1)$, $\bf (H_c)$, $\bf (H_c^2)$ are also satisfied by $\F$ if $\Fb$ satisfies $\bf (H_M)$ and $\bf (H_Q)$. 
\medskip

However, we need to prove Theorem~\ref{th:bitype} first ; to do so we will separately show that $H^{\ell}$ is close to $H^1$ in space (Subsection~\ref{s:space}) and in time (Subsection~\ref{s:time}). Only after that will Section~\ref{s:proof2} be devoted to the proof of Theorem~\ref{th:multitype}.  Then, we will give an application of Theorem~\ref{th:multitype} to random laminations of the disc in Section~\ref{s:application}. Finally, the appendix at the end will propose an alternative hypothesis to $\bf (H_R^{x_0})$ that will be stronger but more convenient to check in practice. 

\section{Proof of Theorem~\ref{th:bitype}}\label{s:proof1}

\subsection{Preliminaries}
A very useful tool when working on Galton-Watson trees is the corresponding {\it size-biased} Galton--Watson tree, which we will introduce in this subsection.  
\subsubsection{Change of measure on $\T^1$}\label{subsec:measurechanget1}
Let us introduce $(W^1_n)_{n\in\N}$ the {\it additive martingale}, where for all $n\in\N$~: 
\begin{equation*}
W_n^1:=\sum_{u\in\T^1,|u|=n} 1. 
\end{equation*}
For any $n\in\N$, denoting by $\scr{F}^1_n$ the $\sigma$-algebra generated by $\{ (u,\ell(u)) \st u\in\T^1, |u|≤n \}$, hypothesis $\bf (H_c)$ and the branching property ensure that $(W^1_n)_{n\in\N}$ is an $\scr{F}^1_n$-martingale. 

Recall that $\zeta^1$ is the law of the progeny on $\bigcup_{n≥0} (\R_+)^n$ of vertices in $\T^1$ (with the convention that $(\R)^0$ is the empty sequence). Let us consider $\hzeta^1$ the probability law with Radon-Nikodym derivative $W_1^1$ with respect to $\zeta^1$, that is such that if $X\sim \zeta^1$ and if $|X|$ denotes the length of $X$, then $\hat{X}\sim \hzeta^1$ if and only if for any bounded function $f:\bigcup_{n≥0} (\R_+)^n\to \R$,
\begin{equation*}
\E\Big[f(\hat{X})\Big]=\E\Big[|X|f(X)\Big].
\end{equation*} 
Notice that almost surely the progeny induced by $\hzeta^1$ is non-empty. Let us introduce a new law $\Ph^*$ on the tree with edge lengths $(\T^1,\ell)$ with self-avoiding distinguished path starting from the root $(w_n)_{n≥0}$, each $w_n$ being at generation $n$. Under $\Ph^*$, we construct  $(\T^1,\ell,(w_n)_{n≥0})$ by induction as follows~:
\begin{itemize}
\item {\bf Initialisation} \\
 Generation $0$ of $\T^1$ is only made up of the root, denoted by $\root$, such that $e(\root)=1$ and $\ell(\root)=0$. We set $w_0=\root$. 
\item {\bf Induction} \\
Let $n≥0$. Suppose that the tree and the spine have been built up to generation $n$. The vertex $w_n$ has progeny according to $\hzeta^1$. Independently, other vertices of generation $n$ give progeny according to $\zeta^1$. The vertex $w_{n+1}$ is chosen uniformly at random among children of $w_{n}$. 
\end{itemize}
This tree is called the size-biased Galton-Watson tree with reproduction law $\zeta^1$. Notice that its construction is such that the $(\ell(w_k))_{k≥1}$ are \iid random variables. We call $\Ph$ the marginal law of $(\T^1,\ell)$ for this construction, and $\Eh$ the associated expectation. We easily adapt the arguments of \cite{lyons-permantle-peres-martingale} to get the following proposition (the only change being that here trees have edge lengths) :  
\begin{proposition}{\bf (\cite{lyons-permantle-peres-martingale})}
Recall that for any $n≥0$, $\scr{F}^1_n$ stands for the sigma-algebra generated by $\{ (u,\ell(u)) \st u\in\T^1, |u|≤n \}$. Then $\Ph_{|\scr{F}^1_n}$ is absolutely continuous with respect to $\P_{|\scr{F}^1_n}$ and is such that
\begin{equation*}
 \frac{\d \Ph}{\d\P}|_{\scr{F}_n^1}=W^1_n. 
\end{equation*}
\end{proposition}
\noindent For the rest of the paper, as the context should ensure that there is no ambiguity, for convenience we will indifferently denote $\Ph$ or $\Ph^*$ by $\Ph$, and by $\Eh$ their associated expectation. A consequence of this proposition is the \textit{many-to-one lemma}, which can be shown by induction~: 
\begin{lemma}\label{lemma:many-to-one}
Let $n\in\N$, $g:\R^{n+1}\to\R$ be a measurable function, and $X_n$ a $\scr{F}^1_n$-measurable random variable. Then, \\
\begin{equation*}
\E\Big[\sum_{|u|=n,u\in\T^1}g(\ell(u_0),\ldots,\ell(u_{n-1}),\ell(u))X_n\Big]=\Eh\Big[g(\ell(w_0),\ldots,\ell(w_{n-1}),\ell(w_n))X_n\Big]. 
\end{equation*}
\end{lemma}
\noindent Notice that applying this lemma, we can re-write hypothesis $\bf(H_1^2)$ as~: 
\begin{itemize}
\item [$(\bf H_1^2)$ :] $y^2\Ph\Big(\ell(w_1)>y \Big)\sto{y\to\infty} 0$. 
\end{itemize}

\subsubsection{Estimates on critical Galton--Watson forests}

Recall that under $\bf (H)$, the forest $\F^1$ is a critical Galton--Watson forest. Recall also that we denote by $u^1(1),\ldots,u^1(n),\ldots$ its vertices taken in the lexicographic order. The following lemma, which is a straightforward consequence of Corollary~2.5.1 of~\cite{duquesne-le-gall}, will allow us to control the shape of~$\F^1$~:

\begin{lemma}\label{lemma:estimatesGW}
Let $\Gamma^1_n:=u^1(n)_0$ be the index of the tree in $\F^1$ to which the n\up{th} vertex of $\F^1$ belongs. Then under $\bf (H)$, for all $\eps>0$, there exist $M,M'>0$ such that for all sufficiently large $n\in\N$,
\begin{equation*}
\P\Big(\Gamma^1_n>M\sqrt{n}\textrm{ or }\max_{0≤i≤n} |u(i)|>M'\sqrt{n} \Big)<\eps.  
\end{equation*}
\end{lemma} 
\begin{proof}
According to Corollary~2.5.1 of~\cite{duquesne-le-gall}, 
\begin{equation*}
\P\Big( \frac{\Gamma^1_n}{\sqrt{n}}>M \Big)\sto{n\to\infty} \P\Big( \sigma L^0_1>M \Big)<\frac{\eps}{3}
\end{equation*}
for $M$ large enough, where $L_1^0$ is the local time at level $0$ at time 1 of a standard Brownian motion. Moreover, 
\begin{equation*}
\P\Big( \frac{\max_{0≤i≤n}{|u(i)|}}{\sqrt{n}}>M' \Big)\sto{n\to\infty} \P\Big(\max_{0≤s≤1}\frac{2}{\sigma}|B_s|>M' \Big)<\frac{\eps}{3}
\end{equation*}
for $M'$ large enough, where $(B_s)_{0≤t≤1}$ is a standard Brownian motion. The union bound concludes the proof. 
\end{proof}

\subsection{Spatial scaling}\label{s:space}

Let for $i\in\N$,  $\varphi(i)$ be the index of $u(i)$ in $\F^1$ if $e(u(i))=1$, or of its parent in $\F^1$  if $e(u(i))=0$~; that is
\begin{equation}\label{eq:defphi}
\varphi(i):=\begin{cases}
k\textrm{, where } u^1(k)=u(i) & \textrm{ if } e(u(i))=1 \\
k\textrm{, where } u^1(k)=\parent{u(i)} & \textrm{ if } e(u(i))=0.
\end{cases}
\end{equation}
In a way, $\varphi$ is the function of re-indexation from $\F$ to $\F^1$. Recall from (\ref{eq:defH}) the definition of $H^1$ and $H^\ell$. We introduce the following proposition, which shows that $\mu$ is the right spatial scale between $H^1$ and $H^\ell$ : 

\begin{proposition}\label{prop:vertical}
Let $(\F,e,\ell)$ be a leafed Galton--Watson forest with edge lengths satisfying hypothesis $\bf (H)$. Then, for all $\eps>0$, 
\begin{equation*}
\P\Big(\max_{1≤ i ≤ n}\Big|H^{\ell}(i)-\mu H^1(\varphi(i))\Big|>\eps \sqrt{n}\Big)\sto{n\to\infty} 0. 
\end{equation*} 
\end{proposition}

\begin{proof}

First of all, let us show that 
\begin{equation}\label{eq:vertical(i)}
\P\Big(\max_{1≤ i ≤ n}\Big|H^{\ell}(i)- h(u^1(\varphi(i)))\Big|>\eps \sqrt{n}\Big)\sto{n\to\infty} 0. 
\end{equation}
According to the definition of $\varphi$, for all $i\in\N$, $H^{\ell}(i)- h(u^1(\varphi(i)))=\ell(u(i))\1{e(u(i))=0}$. Hence,  
\begin{align*}
\P\Big(\max_{1≤ i ≤ n}\Big|H^{\ell}(i)- h(u^1(\varphi(i)))\Big|>\eps \sqrt{n}\Big)&= \P\Big(\max_{1≤ i ≤ n}\ell(u(i))\1{e(u(i))=0}>\eps \sqrt{n}\Big)\\
&≤ \P\Big(\max_{0≤ j ≤ n-1} \max_{\parent{u}=u(j),e(u)=0}\ell(u)>\eps \sqrt{n}\Big), 
\end{align*}
since any $u(i)$ of type $0$ for $1≤i≤n$ is the child of a $u(j)$ for $0≤j≤n-1$. Applying the union bound, we get
\begin{equation*}
\P\Big(\max_{0≤ j ≤ n-1} \max_{\parent{u}=u(j),e(u)=0}|\ell(u)|>\eps \sqrt{n}\Big)≤ \sum_{0≤j≤n-1} \P\Big(\max_{|u|=1,e(u)=0} |\ell(u)|>\eps \sqrt{n}\Big),
\end{equation*}
the last sum tending to $0$ as $n$ tends to infinity, according to hypothesis $\bf (H_0^2)$, thus yielding~(\ref{eq:vertical(i)}). Now, noticing that for all $i\in\N$, $\varphi(i)≤i$, it suffices to show that 
\begin{equation}\label{eq:vertical}
\P\Big(\max_{1≤ i ≤ n}\Big|h(u^1(i))-\mu H^1(i)\Big|>\eps \sqrt{n}\Big)\sto{n\to\infty} 0, 
\end{equation} 
and to combine it with~(\ref{eq:vertical(i)}) to conclude the proof of the proposition. To this end, we will use a method employed in the proof that appears in Section~3 of~\cite{durrett-kesten-waymire}, which is built in 3 steps \,-- but we will have to adjust some parts. We emphasise that until the end of the proof, all considered vertices are in $\T^1$ or $\F^1$, and that the lexicographical order $u^1$ is also taken in $\T^1$ or $\F^1$.\\
The first step is to show that~: 
\begin{equation}\label{eq:maxbn}
\P\Big( \exists i≤ n \st \ell(u^1(i))>s_n \Big)\;\sto{n\to\infty}\;0,
\end{equation}
where $(s_n)_{n\in\N}$ is any sequence of real numbers such that 
\begin{equation*}
 s_n\substack{ \\ =\\ n\to\infty} o(n^{\frac{1}{2}})\textrm{\hspace{1cm} and\hspace{1cm} }\Ph(\ell(w_1)>s_n)\;\substack{\\= \\n\to\infty}\; o(\frac{1}{n}) 
\end{equation*}
(such a sequence exists thanks to condition $\bf (H_1^2)$). According to Lemma~\ref{lemma:estimatesGW}, for any $\eps>0$, and then $M,M',n$ large enough, 
\begin{equation*}
\P\Big( \exists i≤ n \st \ell(u^1(i))>s_n \Big) ≤ \eps + \P\Big( \exists i≤ n\st \ell(u^1(i))>s_n, \Gamma^1_n<M\sqrt{n}, \max_{1≤i≤n} |u^1(i)|<M'\sqrt{n} \Big). 
\end{equation*}
Discussing on which tree $u^1(i)$ belongs to and on its generation we get~: 
\begin{align*}
\P\Big( \exists i≤ n \st \ell(u^1(i))>s_n \Big) &≤ \eps + \sum_{k=1}^{\fl{M\sqrt{n}}}\E\Big[ \sum_{l=1}^{\fl{M'\sqrt{n}}} \sum_{|u|=l,u\in\T^1} \1{\ell(u)>s_n} \Big] \\
&≤ \eps + MM'n \Ph\Big( \ell(w_1)>s_n \Big)=\eps+o(1),
\end{align*}
where we used the many-to-one lemma (Lemma~\ref{lemma:many-to-one}) between lines 1 and 2 and then used the second property of $(s_n)_{n\in\N}$. This proves~(\ref{eq:maxbn}). \\
Now, for all $n\in\N$, we let $v_n:=n^{3/8}$. The second step of our proof is to show that~: 
\begin{equation}\label{eq:max_an}
\P\Big( \exists\, i≤ n \st \ell(u^1(i))>v_n \textrm{ and } \exists\, u\vdash u^1(i),\, \ell(u)>v_n \Big)\;\sto{n\to\infty}\;0.
\end{equation}
Once again, using Lemma~\ref{lemma:estimatesGW}, for $M$, $M'$ and then $n$ large enough,~(\ref{eq:max_an}) is smaller than
\begin{equation*}
\eps +\P\Big( \exists\:i≤n, \ell(u^1(i))>v_n \textrm{ and } \exists\, u\vdash u^1(i),\, \ell(u)>v_n ,|u^1(i)|<M'\sqrt{n}, \Gamma^1_n<M\sqrt{n}\Big),
\end{equation*}
and once again discussing on which trees the vertices $u^1(i)$ belong to and on their generation, the latter sum is smaller than
\begin{align*}
\sum_{l=1}^{\fl{M\sqrt{n}}}\sum_{k=0}^{\fl{M'\sqrt{n}}}\E\Big[\sum_{|u|=k} \1{\ell(u)>v_n \textrm{ and } \exists\, v\vdash u,\, \ell(v)>v_n}\Big]&=M\sqrt{n}\sum_{k=0}^{\fl{M'\sqrt{n}}}\Ph\Big( \ell(w_k)>v_n\textrm{ and } \exists\, l<k,\, \ell(w_l)>v_n \Big) \\
&≤M\sqrt{n}\sum_{k=0}^{\fl{M'\sqrt{n}}}\Ph\Big( \ell(w_k)>v_n\Big)\sum_{l=0}^{k-1}\Ph\Big(\ell(w_l)>v_n \Big) \\
&≤MM'^2n^{3/2}\Ph\Big( \ell(w_1) >v_n \Big)^2\\
&\sto[=]{n\to\infty} o(v_n^{-4} n^{3/2})\sto[=]{n\to\infty} o(1),
\end{align*}
yielding~(\ref{eq:max_an}). \\
To sum up, we can consider now that for $n$ large enough, on every path in $\F^1$ there is at most one $u$ such that $\ell(u)>v_n$, and that for that $u$, necessarily, $\ell(u)<s_n$. More precisely, we let for all $n\in\N$, $u\in\F^1$,
\begin{equation*}
\ell^{(n)}(u):=\ell(u)\1{\ell(u)<v_n},
\end{equation*}
and we can write using~(\ref{eq:maxbn}) and~(\ref{eq:max_an})~: 
\begin{equation}\label{eq:Hencadrement}
\P\Big(\forall i≤ n,\;\sum_{u\vdash u^1(i)} \ell^{(n)}(u) -\mu H^1(i) ≤ h(u^1(i))-\mu H^1(i)≤ \sum_{u\vdash u^1(i)} \ell^{(n)}(u) +s_n -\mu H^1(i) \Big)\sto{n\to\infty} 1. 
\end{equation}
Thus, since $s_n\sto[=]{n\to\infty} o(n^{\frac{1}{2}})$, the triangle inequality yields that
\begin{equation*}
\P\Big(\exists\, i≤ n,\; \Big|\sum_{u\vdash u^1(i)} \ell^{(n)}(u) -\mu H^1(i) +s_n \Big|>\eps\sqrt{n}\Big)≤\P\Big(\exists\, i≤ n,\; \Big|\sum_{u\vdash u^1(i)} \ell^{(n)}(u) -\mu H^1(i) \Big|>\frac{\eps}{2}\sqrt{n}\Big),
\end{equation*}
for $n$ large enough, and thus we just have to show that 
\begin{equation}\label{eq:Hlimite}
\P\Big(\exists\, i≤ n,\; \Big|\sum_{u\vdash u^1(i)} \ell^{(n)}(u) -\mu H^1(i) \Big|>\eps\sqrt{n}\Big)\sto{n\to\infty} 0,
\end{equation}
and to use~(\ref{eq:Hencadrement}) to get~(\ref{eq:vertical}). This will be the last step of our proof. \\
Actually, once again using Lemma~\ref{lemma:estimatesGW}, and then applying the many-to-one lemma (Lemma~\ref{lemma:many-to-one}), 
\begin{align}\label{eq:decomplemme1}
\P\Big(\exists\, i≤ n,\; \Big|\sum_{u\vdash u^1(i)} \ell^{(n)}(u) -\mu H^1(i) \Big|>\eps\sqrt{n}\Big)&≤ \eps + M\sqrt{n}\sum_{k=0}^{\fl{M'\sqrt{n}}}\E\Big[ \sum_{|u|=k} \1{ |\sum_{v\vdash u} \ell^{(n)}(v)-\mu k|>\eps\sqrt{n} } \Big] \nonumber\\
&=\eps + M\sqrt{n}\sum_{k=0}^{\fl{M'\sqrt{n}}}\Ph\Big( |\sum_{i=1}^{k} \ell^{(n)}(w_i)-\mu k|>\eps\sqrt{n} \Big). 
\end{align}
Let us focus on the general term in the sum, for any $k≤ M'\sqrt{n}$. First, notice that $\mu=\E\left[ \sum_{|u|=1} \ell(u) \right]=\Eh\left[ \ell(w_1) \right]$ by the many-to-one lemma. Hence, $\Eh\Big[ \ell^{(n)}(w_1) \Big]\sto{n\to\infty}\mu$ by monotone convergence. Take $n$ large enough such that $\Big|\Eh\Big[ \ell^{(n)}(w_1) \Big]-\mu\Big|≤\frac{\eps}{2M}$ . We then have for any $k≤ M\sqrt{n}$,
\begin{align*}
\Ph\Big( \Big|\sum_{i=1}^{k} \ell^{(n)}(w_i)-\mu k\Big|>\eps\sqrt{n} \Big)&≤ \Ph\Big( \Big|\sum_{i=1}^{k} \ell^{(n)}(w_i)-\Eh\Big[ \ell^{(n)}(w_1) \Big]\Big|+k\Big|\Eh\Big[ \ell^{(n)}(w_1) \Big]-\mu\Big|>\eps\sqrt{n} \Big) \\
&≤ \Ph\Big( \Big|\sum_{i=1}^{k}\Big( \ell^{(n)}(w_i)- \Eh\Big[ \ell^{(n)}(w_1) \Big]\Big)\Big|>\frac{\eps}{2}\sqrt{n} \Big). 
\end{align*}
Now, for $i≥1$, let us set $X_i:= \ell^{(n)}(w_1)- \Eh\Big[ \ell^{(n)}(w_1) \Big]$. Notice that the $(X_i)_{i≥1}$ are \iid centred random variables. We have for all $r\in\brint{2;8}$, 
\begin{align*}
\xi_r(n):=\Eh\Big[|X_i|^r\Big]&= \Eh\Big[ \Big| \ell^{(n)}(w_1)- \Eh\Big[ \ell^{(n)}(w_1) \Big]\Big|^r \Big] \\
		&= r\int_{0}^{+\infty} y^{r-1}\Ph\Big( |\ell^{(n)}(w_1) - \Eh\Big[ \ell^{(n)}(w_1) \Big]|>y \Big)\d y \\
		&≤ r\int_{0}^{+\infty} y^{r-1}\Ph\Big( |\ell^{(n)}(w_1)|>y-|\Eh\Big[ \ell(w_1) \Big]| \Big)\d y \\
		&= r\int_{0}^{v_n+\Eh[\ell(w_1)]} y^{r-1}\Ph\Big( \ell(w_1)>y-\Eh\Big[ \ell(w_1) \Big] \Big)\d y,  \\
\end{align*}
where we used the triangle inequality at line 3, and then the fact that $\ell^{(n)}≤v_n$. Hypothesis $\bf(H_1^2)$ allows us to consider $M_0:=\max_{y>1} \left( y^2\Ph\left(\ell(w_1)>y-\Eh[\ell(w_1)]\right)\right)$, and then cutting the integral at $y=1$ we get 
\begin{align*}
\xi_r(n)&≤ r \Big( 1 + \int_{1}^{v_n+\Eh[\ell(w_1)]} y^{r-3}M_0 \d y \Big) \\
		&≤ c(r) v_n^{r-2}\vee \ln(v_n)≤ c(r) n^{\frac{3(r-2)}{8}}\ln(n),
\end{align*}
where $c(r)$ is a suitable constant. Thus we can write, the $X_i$ being independent,
\begin{align*}
\Eh\Big[\Big(\sum_{i=1}^{k} X_i\Big)^8 \Big]&=\sum_{\substack{0≤ i_1, \ldots, i_k ≤ 8 \\ i_1+\cdots+i_k=8}}\frac{8!}{i_1!\ldots i_k!}\prod_{j=1}^{k}\Eh\Big[ {X_j}^{i_j}\Big] \\
&=\sum_{\substack{0≤ i_1, \ldots, i_k ≤ 8 \\ i_1+\cdots+i_k=8 \\ i_1,\ldots,i_k\neq 1}}\frac{8!}{i_1!\ldots i_k!}\prod_{j=1}^{k}\Eh\Big[ {X_j}^{i_j}\Big], \\
\end{align*}
where between lines 1 and 2 we used the fact that $\Eh[X_i]=\Eh\Big[  \ell^{(n)}(w_j)- \Eh\Big[ \ell^{(n)}(w_1) \Big] \Big]=0$. Now we just have to regroup common patterns on $i_1,\ldots,i_k$ and we get that,
\begin{align*}
\Eh\Big[\Big(\sum_{i=1}^{k} X_i\Big)^8 \Big]&≤c\Big[k^4 \xi_2(n)^4+k^3\Big(\xi_2(n)^2\xi_4(n) +\xi_3(n)^2\xi_2(n)\Big) \\
&+k^2\Big(\xi_4(n)^2+\xi_6(n)\xi_2(n)+\xi_5(n)\xi_3(n)\Big)+k\xi_8(n)\Big] \\
&≤ c' n^{3-\frac{1}{4}}\ln(n), 
\end{align*}
where we used the fact that $k≤ M\sqrt{n}$ in the last inequality, and where $c$ and $c'$ are suitable constants. Applying Markov's inequality yields 
\begin{align*}
\Ph\Big( \Big|\sum_{i=1}^{k}\Big( \ell^{(n)}(w_k)- \Eh\Big[ \ell^{(n)}(w_1) \Big]\Big)\Big|>\frac{\eps}{2}\sqrt{n} \Big)
&≤ (\frac{2}{\eps})^8n^{-4} \Eh\Big[\Big(\sum_{i=1}^{k}\Big( \ell^{(n)}(w_k)- \Eh\Big[ \ell^{(n)}(w_1) \Big]\Big)\Big)^8 \Big] \\
&=(\frac{2}{\eps})^8n^{-4} \Eh\Big[\Big(\sum_{i=1}^{k} X_i\Big)^8 \Big] \\
&≤ (\frac{2}{\eps})^8c' n^{-1-\frac{1}{4}}\ln(n), 
\end{align*}
and when using this in~(\ref{eq:decomplemme1}), we finally get that 
\begin{equation*}
\P\Big(\exists\, i≤ n,\; \Big|\sum_{u\vdash u^1(i)} \ell^{(n)}(u) -\mu H^1(i) \Big|>\eps\sqrt{n}\Big)\sto{n\to\infty} 0,
\end{equation*}
which proves~(\ref{eq:Hlimite}) and concludes the proof. \\
\end{proof}
\bigskip

\subsection{Time scaling}\label{s:time}
In the previous subsection, we showed that the renormalised height function of a leafed Galton--Watson forest with edge lengths was "close in space" to the height process of $\F^1$ a simple Galton--Watson forest. Now, we want to prove that they can also be "close in time" up to a scaling. Recall from (\ref{eq:defphi}) that $\varphi$ is the function of re-indexation from $\F$ to $\F^1$. 

\begin{proposition}\label{prop:horizontal}
Recall $\bf (H)$ from Subsection~\ref{s:introleafed}, and recall that $m=\E[\nu]$. Under $\bf (H)$, the function $(\varphi(\fl{ns})/n)_{s>0}$ converges in probability to $(m^{-1}s)_{s>0}$ as $n$ tends to infinity, for the topology of uniform convergence over compact sets. 
\end{proposition}

\begin{proof}
We  only need to prove : 
\begin{equation}\label{eq:cvm}
\frac{\varphi(n)}{n}\substack{\P \\ \longrightarrow \\ n\to \infty} m^{-1}.
\end{equation}
Indeed, this would imply the convergence in probability of the finite-dimensional marginal distributions of $(\varphi(\fl{ns})/n)_{s≥0}$ towards those of $(m^{-1}s)_{s≥0}$. Since $(m^{-1}s)_{s≥0}$ is a continuous function, and since the $(\varphi(\fl{ns})/n)_{s≥0}$ are non-decreasing functions for $n≥1$, a standard argument due to Dini would yield the convergence in law on Skorokhod's space. The limit process $(m^{-1}s)_{s≥0}$ being deterministic and continuous, this convergence would also holds in probability on the topology of uniform convergence over compact sets, as required. \\

Let $\psi$ be the function of re-indexation from $\F^1$ to $\F$, that is we set for all $n\in\N$, \break $\psi(n):=\#\{ u\in\F \st u\prec u^1(n) \} $. Somehow, $\psi$ can be seen as the inverse function of $\varphi$. Just as in the proof of Proposition 6 in~\cite{miermont}, notice that we have for all $n\in\N$ 
\begin{equation*}
\psi(n)=\sum_{k=0}^{n-1} \nu(u^1(k))-\underbrace{\sum_{k=0}^{n-1}\#\{ u\in\F \st \parent{u}=u^1(k), u^1(n)\prec u \}}_{:=R(n)}, 
\end{equation*}
that is $\psi(n)$ is the sum of the number of children of each vertex lexicographically smaller than $u^1(n)$, minus the children which come lexicographically after $u^1(n)$. 

\begin{figure}[H]
\begin{tikzpicture}[line cap=round,line join=round,>=triangle 45,x=1.5cm,y=0.5cm]
\clip(7,-1.0) rectangle (20.8,12);
\draw (12.5,0.0)-- (11.0,1.5);
\draw (12.5,0.0)-- (13.0,1.5);
\draw (11.0,1.5)-- (10.5,3.0);
\draw (13.0,1.5)-- (12.5,3.0);
\draw (13.0,1.5)-- (14.0,3.0);
\draw (14.0,3.0)-- (14.0,4.0);
\draw (14.0,3.0)-- (15.5,4.5);
\draw (12.5,3.0)-- (12.0,4.5);
\draw (12.5,3.0)-- (13.0,5.0);
\draw (12.0,4.5)-- (11.0,6.0);
\draw (12.0,4.5)-- (12.0,6.5);
\draw (12.5,3.0)-- (11.0,4.0);
\draw (13.0,5.0)-- (12.5,7.0);
\draw (13.0,5.0)-- (13.5,7.0);
\draw (13.0,5.0)-- (14.0,6.5);
\draw (12.5,3.0)-- (14.0,5.0);
\draw (12.5,7.0)-- (11.5,9.0);
\draw (12.5,7.0)-- (12.5,8.5);
\draw (12.5,8.5)-- (13.0,10.0);
\draw (12.5,8.5)-- (11.5,10.5);
\draw (11.0,1.5)-- (11.7,2.6);
\draw (11.0,1.5)-- (10.1,2.4);
\draw (11.0,1.5)-- (11.0,2.5);
\draw (13.0,1.5)-- (12.1,2.4);
\draw (13.0,1.5)-- (13.1,2.6);
\draw (13.0,1.5)-- (14.0,2.1);
\draw (14.0,3.0)-- (13.4,3.6);
\draw (14.0,3.0)-- (14.7,3.2);
\draw (14.0,3.0)-- (14.3,3.8);
\draw (15.5,4.5)-- (15.3,5.2);
\draw (14.0,4.0)-- (14.6,4.8);
\draw (14.0,5.0)-- (13.9,5.6);
\draw (14.0,5.0)-- (14.5,5.7);
\draw (11.0,4.0)-- (11.5,4.9);
\draw (11.0,4.0)-- (10.5,4.9);
\draw (13.0,5.0)-- (12.5,6.0);
\draw (13.0,5.0)-- (13.0,6.5);
\draw (13.0,5.0)-- (13.6,5.4);
\draw (14.0,6.5)-- (13.9,7.4);
\draw (14.0,6.5)-- (14.5,7.6);
\draw (13.5,7.0)-- (13.5,8.2);
\draw (12.5,7.0)-- (12.9,7.8);
\draw (12.5,7.0)-- (12.2,8.1);
\draw (12.5,7.0)-- (11.4,7.9);
\draw (12.0,4.5)-- (11.7,5.8);
\draw (11.0,6.0)-- (10.5,7.0);
\draw (11.0,6.0)-- (11.4,7.0);
\draw (12.5,8.5)-- (12.4,9.6);
\draw (13.0,10.0)-- (13.6,10.8);
\draw (13.0,10.0)-- (12.4,11.3);
\draw (11.5,10.5)-- (10.9,11.5);
\draw (11.5,10.5)-- (11.9,11.7);
\draw (11.5,9.0)-- (11.0,10.2);
\draw (12.5,3.0)-- (13.2,4.5);
\draw (13.0,1.5)-- (15.0,2.2);
\draw (12.5,7.0)-- (13.2,7.6);
\draw (16.5,3.1) node[anchor=north west] {$R(n)$};
\draw (16.5,1.9) node[anchor=north west] {$\psi(n)$};
\draw [fill=green] (12.5,0.0) circle (2pt);
\draw [fill=green] (11.0,1.5) circle (2.0pt);
\draw [fill=green] (13.0,1.5) circle (2pt);
\draw [fill=green] (10.5,3.0) circle (2pt);
\draw [fill=green] (12.5,3.0) circle (2pt);
\draw [color=red] (14.0,3.0)-- ++(-2pt,-2pt) -- ++(4.0pt,4.0pt) ++(-4.0pt,0) -- ++(4.0pt,-4.0pt);
\draw [color=black] (14.0,4.0) circle (2pt);
\draw [color=black] (15.5,4.5) circle (2pt);
\draw [fill=green] (12.0,4.5) circle (2pt);
\draw [fill=green] (13.0,5.0) circle (2pt);
\draw [fill=green] (11.0,6.0) circle (2pt);
\draw [fill=green] (12.0,6.5) circle (2pt);
\draw [fill=green] (11.0,4.0) circle (2pt);
\draw [fill=green] (12.5,7.0) circle (2pt);
\draw [color=red] (13.5,7.0)-- ++(-2pt,-2pt) -- ++(4.0pt,4.0pt) ++(-4.0pt,0) -- ++(4.0pt,-4.0pt);
\draw [color=red] (14.0,6.5)-- ++(-2pt,-2pt) -- ++(4.0pt,4.0pt) ++(-4.0pt,0) -- ++(4.0pt,-4.0pt);
\draw [color=red] (14.0,5.0)-- ++(-2pt,-2pt) -- ++(4.0pt,4.0pt) ++(-4.0pt,0) -- ++(4.0pt,-4.0pt);
\draw [fill=green] (11.5,9.0) circle (2pt);
\draw [fill=blue] (12.5,8.5) circle (2pt);
\draw (12.5,8.5) node[anchor=west] {\scriptsize $u(n)$};
\draw [color=black] (13.0,10.0) circle (2pt);
\draw [color=black] (11.5,10.5) circle (2pt);
\draw [fill=green] (11.7,2.6) circle (2pt);
\draw [fill=green] (10.1,2.4) circle (2pt);
\draw [fill=green] (11.0,2.5) circle (2pt);
\draw [fill=green] (12.1,2.4) circle (2pt);
\draw [color=red] (13.1,2.6)-- ++(-2pt,-2pt) -- ++(4.0pt,4.0pt) ++(-4.0pt,0) -- ++(4.0pt,-4.0pt);
\draw [color=red] (14.0,2.1)-- ++(-2pt,-2pt) -- ++(4.0pt,4.0pt) ++(-4.0pt,0) -- ++(4.0pt,-4.0pt);
\draw [color=black] (13.4,3.6) circle (2pt);
\draw [color=black] (14.7,3.2) circle (2pt);
\draw [color=black] (14.3,3.8) circle (2pt);
\draw [color=black] (15.3,5.2) circle (2pt);
\draw [color=black] (14.6,4.8) circle (2pt);
\draw [color=black] (13.9,5.6) circle (2pt);
\draw [color=black] (14.5,5.7) circle (2pt);
\draw [fill=green] (11.5,4.9) circle (2pt);
\draw [fill=green] (10.5,4.9) circle (2pt);
\draw [fill=green] (12.5,6.0) circle (2pt);
\draw [color=red] (13.0,6.5)-- ++(-2pt,-2pt) -- ++(4.0pt,4.0pt) ++(-4.0pt,0) -- ++(4.0pt,-4.0pt);
\draw [color=red] (13.6,5.4)-- ++(-2pt,-2pt) -- ++(4.0pt,4.0pt) ++(-4.0pt,0) -- ++(4.0pt,-4.0pt);
\draw [color=black] (13.9,7.4) circle (2pt);
\draw [color=black] (14.5,7.6) circle (2pt);
\draw [color=black] (13.5,8.2) circle (2pt);
\draw [color=red] (12.9,7.8)-- ++(-2pt,-2pt) -- ++(4.0pt,4.0pt) ++(-4.0pt,0) -- ++(4.0pt,-4.0pt);
\draw [fill=green] (12.2,8.1) circle (2pt);
\draw [fill=green] (11.4,7.9) circle (2pt);
\draw [fill=green] (11.7,5.8) circle (2pt);
\draw [fill=green] (10.5,7.0) circle (2pt);
\draw [fill=green] (11.4,7.0) circle (2pt);
\draw [color=black] (12.4,9.6) circle (2pt);
\draw [color=black] (13.6,10.8) circle (2pt);
\draw [color=black] (12.4,11.3) circle (2pt);
\draw [color=black] (10.9,11.5) circle (2pt);
\draw [color=black] (11.9,11.7) circle (2pt);
\draw [fill=green] (11.0,10.2) circle (2pt);
\draw [color=red] (13.2,4.5)-- ++(-2pt,-2pt) -- ++(4.0pt,4.0pt) ++(-4.0pt,0) -- ++(4.0pt,-4.0pt);
\draw [color=red] (15.0,2.2)-- ++(-2pt,-2pt) -- ++(4.0pt,4.0pt) ++(-4.0pt,0) -- ++(4.0pt,-4.0pt);
\draw [color=red] (13.2,7.6)-- ++(-2pt,-2pt) -- ++(4.0pt,4.0pt) ++(-4.0pt,0) -- ++(4.0pt,-4.0pt);
\draw [color=red] (16.5,2.1)-- ++(-2pt,-2pt) -- ++(4.0pt,4.0pt) ++(-4.0pt,0) -- ++(4.0pt,-4.0pt);
\draw [fill=green] (16.5,0.9) circle (2pt);
\draw [color=red] (13.2,4.5) circle (2.5pt);
\draw [color=red] (15.0,2.2) circle (2.5pt);
\draw [color=red] (13.2,7.6) circle (2.5pt);
\draw [color=red] (16.5,2.1) circle (2.5pt);
\draw [color=red] (12.9,7.8) circle (2.5pt);
\draw [color=red] (13.0,6.5) circle (2.5pt);
\draw [color=red] (13.6,5.4) circle (2.5pt);
\draw [color=red] (13.1,2.6) circle (2.5pt);
\draw [color=red] (14.0,2.1) circle (2.5pt);
\draw [color=red] (13.5,7.0) circle (2.5pt);
\draw [color=red] (14.0,6.5) circle (2.5pt);
\draw [color=red] (14.0,5.0) circle (2.5pt);
\draw [color=red] (14.0,3.0) circle (2.5pt);
\end{tikzpicture}
\caption{Vertices counted in $\psi(n)$ -- Vertices counted in $R(n)$. }
\label{f:Rn}
\end{figure}
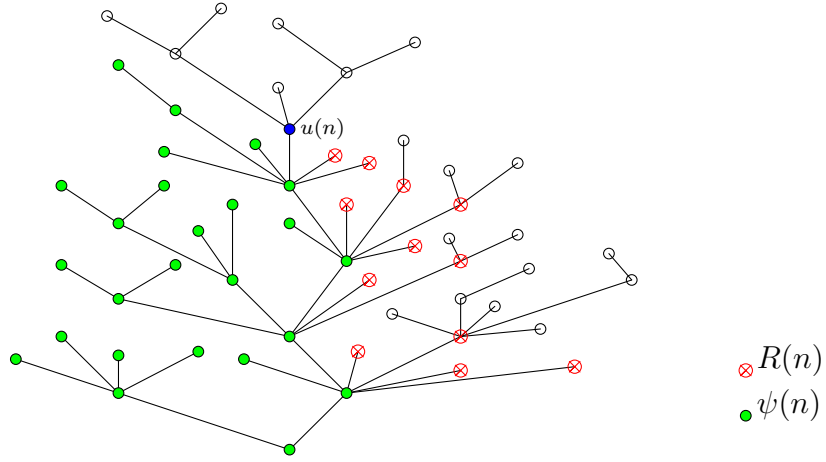

\noindent We want to show that for all $\eps >0$, 
\begin{equation}\label{eq:hortoshow}
\P\Big( R(n)>\eps n \Big)\sto{n\to\infty} 0. 
\end{equation}
Indeed, suppose~(\ref{eq:hortoshow}) is proved, then applying the weak law of large numbers to $\sum_{k=0}^{n-1}\nu(u^1(k))$ yields 
\begin{equation*}
\frac{\psi(n)}{n}=\frac{\sum_{k=0}^{n-1}\nu(u^1(k))}{n}-\frac{R(n)}{n}\substack{\P \\ \longrightarrow \\ n\to\infty} m, 
\end{equation*}
and noticing that for all $n\in\N$, $\psi(\varphi(n))=n$, this would imply (\ref{eq:cvm}), which would conclude the proof as explained previously. \\ First of all, we have obviously 
\begin{equation*}
R(n) ≤ \sum_{k=0}^{n-1} \left(\nu(u^1(k))\1{\#\{ u\in\F \st \parent{u}=u^1(k), u^1(n)\prec u \}\neq0}\right). 
\end{equation*}
However, for all $k\in\N$, it is necessary that $u^1(k)\vdash u^1(n)$ for \mbox{$\{ u\in\F \st \parent{u}=u^1(k), u^1(n)\prec u \}$} not to be empty. Thus, 
\begin{equation*}
R(n)≤\sum_{u\vdash u^1(n)} \nu(u),  
\end{equation*}
and therefore
\begin{equation*}
\P\Big( R(n)>\eps n \Big)≤\P\Big(\sum_{u\vdash u^1(n)} \nu(u)>\eps n\Big). 
\end{equation*}
Now, notice that since $\E\Big[ (\nu^1)^2 \Big]<\infty$ (according to $\bf (H_c^2)$), there exists a sequence $(c_n)_{n≥1}$ going to infinity such that
\begin{equation*}
c_n\sto[=]{n\to\infty} o(\sqrt{n})\textrm{ and }\P\Big( \nu^1>c_n \Big)\sto[=]{n\to\infty} o(\frac{1}{n}). 
\end{equation*}
Then, by the union bound, 
\begin{equation*}
\P\Big( \exists i<n \st \nu^1(u^1(i))>c_n \Big)≤ n \P\Big( \nu^1>c_n \Big) \;\sto{n\to\infty}\;0,
\end{equation*}
and therefore we have for all $\eps'>0$, for $n$ large enough,
\begin{equation*}
\P\Big( R(n)>\eps n \Big)≤\eps'+\P\Big(\sum_{u\vdash u^1(n)} \nu(u)>\eps n,\:\max_{i<n}\nu^1(u(i))<c_n\Big).
\end{equation*}
Moreover according to our estimate on Galton--Watson forests in Lemma~\ref{lemma:estimatesGW}, we notice that for all $\eps'>0$, for $M$ and $n$ large enough, 
\begin{align}\label{eq:beginninghorizontal}
\P\Big( R(n)>\eps n \Big) &≤ 2\eps' + \P\Big( \sum_{u\vdash u^1(n)} \nu(u)>\eps n,\:\max_{i<n}\nu^1(u^1(i))<c_n,\:|u^1(n)|≤ M\sqrt{n} \Big) \nonumber \\ 
&≤2\eps' + \frac{1}{\eps n}\underbrace{\E\Big[  \Big(\sum_{u\vdash u^1(n)}\nu(u)\Big)\1{\max_{i<n}\nu^1(u^1(i))<c_n,\:|u^1(n)|≤ M\sqrt{n}} \Big]}_{:=A_n}, 
\end{align}
and so it is sufficient to show that the expectation denoted by $A_n$ is $o(n)$ to get~(\ref{eq:hortoshow}). \\ 
To this end,  let us set $S_0:=0$ and for all $k≥1$, $S_k:=\sum_{i=0}^{k-1} (\nu^1(u^1(i))-1)$ . The sequence $(S_k)_{k≥0}$ is the \textit{Lukasiewicz path} of $\F^1$, a centred random walk, see Part 1.1 of \cite{le-gall-rt}. Then, as explained in the proof of Corollary~2.2 of \cite{le-gall-le-jan}, we have that for all $k<\N$,
\begin{equation*}
u^1(k)\vdash u^1(n) \Longleftrightarrow S_k=\min_{k≤l≤n} S_l. 
\end{equation*} 
Hence, we can write
\begin{align*}
A_n&=\E\Big[  \Big(\sum_{u\vdash u^1(n)} \nu(u)\Big)\1{\max_{0≤i<n}\nu^1(u^1(i))<c_n,\:|u^1(n)|≤ M\sqrt{n}} \Big]\\
&=\E\Big[  \Big(\sum_{k=0}^{n-1} \nu(u^1(k))\1{S_k=\min_{k≤l≤n} S_l}\Big)\1{\max_{0≤i<n}\nu^1(u^1(i))<c_n,\:\#\{ 0≤i<n \st S_i=\min_{i≤j≤n} S_j \}≤M\sqrt{n}} \Big].  
\end{align*}
Now we let $(\hat{S}^n_k)_{0≤ k ≤ n}=(S_n-S_{n-k})_{0≤ k ≤ n}$ be the time-reverse from time $n$ version of $(S_k)_{0≤k≤n}$. Re-indexing the sum from $n-1$ to $0$ and using the fact $S_{n-k}=\min_{n-k≤l≤n} S_l$ if and only if $\hat{S}_k=\max_{0≤l≤k} \hat{S}_l$ yields
\begin{align*}
A_n&=\E\Big[  \Big(\sum_{k=1}^{n} \nu(u^1(n-k))\1{\hat{S}_k=\max_{0≤l≤k} \hat{S}_l}\Big)\1{\max_{0<i≤n}\nu^1(u^1(n-i))<c_n,\:\#\{ 0<i≤ n\st \hat{S}_i=\max_{0≤j≤i} \hat{S}_j \}≤M\sqrt{n}} \Big] \\
&=\E\Big[  \Big(\sum_{k=1}^{n} \nu(u^1(k))\1{S_k=\max_{0≤l≤k} S_l}\Big)\1{\max_{0<i≤n}\nu^1(u^1(i))<c_n,\:\#\{ 0≤i<n \st S_i=\max_{0≤j≤i} S_j \}≤M\sqrt{n}} \Big]. 
\end{align*}
where in the last equality we used the fact that $\Big((\hat{S}^n_k)_{0≤ k ≤ n}, (\nu(u^1(n-k)))_{0≤k≤n}\Big)$ has the same law than $\Big((S_k)_{0≤k≤n},(\nu(u^1(k)))_{0≤k≤n}\Big)$. Let 
\begin{equation*}
\tau_1=\inf \{k≥1 \st S_k>0 \}\textrm{ and }\forall i\in\N,\: \tau_{i+1}=\inf \{k>\tau_i \st S_k>\max_{l<k} S_l \} 
\end{equation*}
be the stopping times at which record high are achieved, we have 
\begin{align*}
A_n&=\E\Big[  \Big(\sum_{k≥1} \nu(u^1(\tau_k))\1{\tau_k≤n}\Big)\1{\max_{i≤n}\nu^1(u^1(i))<c_n,\tau_{\cl{M\sqrt{n}}}≥n} \Big] \\
&≤ \sum_{k=1}^{\fl{M\sqrt{n}}} \E\Big[  \nu(u^1(\tau_k))\1{\nu^1(u^1(\tau_k))<c_n} \Big]. 
\end{align*}
Applying Markov's strong property to stopping times $\tau_1,\ldots,\tau_{\fl{M\sqrt{n}}}$, we obtain
\begin{equation*}
A_n ≤M\sqrt{n}\E\Big[ \nu(u^1(\tau_1))\1{\nu^1(u^1(\tau_1))<c_n}\Big]. 
\end{equation*}
Let us estimate $\E\Big[ \nu(u^1(\tau_1))\1{\nu^1(u^1(\tau_1))<c_n}\Big]$~: 

\begin{align*}
\E\Big[ \nu(u^1(\tau_1))\1{\nu^1(u^1(\tau_1))<c_n}\Big]&=\E\Big[\sum_{k≥1} \nu(u^1(k))\1{\forall 0≤i≤k-1,\, S_i≤0,\, S_{k-1}+\nu^1(u^1(k))-1>0}\1{\nu^1(u^1(k))<c_n}\Big] \\
&≤\E\Big[\sum_{k≥1} \nu(u^1(k))\1{\forall 0≤i≤k-1,\, S_i≤0,\, S_{k-1}+c_n-1>0}\Big] \\
&≤\E\Big[ \nu \Big]\sum_{k≥1}\E\Big[\1{\forall 0≤i≤k-1,\, S_i≤0,\, S_{k-1}+c_n-1>0}\Big] \\
&=m\E\Big[ \sum_{k=0}^{\tau_1-1} \1{S_k>-c_n+1} \Big]
\end{align*}
Proceeding as in Section~2 of~\cite{biggins}, we have
\begin{equation*}
\E\Big[ \sum_{k=0}^{\tau_1-1} \1{S_k>-c_n+1} \Big]=\int_{0}^{c_n-1}U^{-}(\d x), 
\end{equation*}
where $U^{-}$ is the renewal measure corresponding to the weak descending ladders heights of $(S_n)_{n≥0}$. The renewal theorem (see p.\ 360 in~\cite{feller}) ensures us that there exists a constant $c'>0$ such that
\begin{equation*}
\int_{0}^{c_n-1}U^{-}(\d x)<c'(1+c_n-1). 
\end{equation*}
Hence, 
\begin{align*}
A_n≤(M\sqrt{n})m(c'c_n)\sto[=]{n\to\infty} o(n), 
\end{align*}
which is what we wanted in equation~(\ref{eq:beginninghorizontal}).  
\end{proof}

\subsection{Conclusion of the proof of Theorem~\ref{th:bitype}}\label{s:conclusion1}
To conclude the proof of Theorem~\ref{th:bitype} (i), we just have to use the convergence of $s\mapsto H^1(\fl{ns})/n^{1/2}$ together with Propositions~\ref{prop:vertical} and~\ref{prop:horizontal} to get the convergence of $s\mapsto H^{\ell}(\fl{ns})/n^{1/2}$. \\
\medskip 

\noindent {\it Proof of Theorem~\ref{th:bitype} (i).} Recall that the forest $\F^1$ is a non-trivial critical Galton--Watson forest with finite variance. Then, 
\begin{equation}\label{eq:cvsimple}
\Big(n^{-1/2}H^1(\fl{ns}) \Big)_{s≥0}\sto[\Longrightarrow]{n\to\infty}\left(\frac{2}{\sigma}B_s\right)_{s≥0}
\end{equation}
for the Skorokhod topology on the space $\mathbb{D}(\R_+,\R)$ (this is Theorems~2.3.1 and~2.3.2 of~\cite{duquesne-le-gall} for example). Now, composing $s\mapsto H^1(\fl{ns})$ with $s\mapsto \varphi(\fl{ns})/n$, Proposition~\ref{prop:horizontal} ensures that 
\begin{equation*}
\Big(n^{-1/2}\mu H^1(\varphi(\fl{ns})\Big))_{s≥0} \sto[\Longrightarrow]{n\to\infty} \left(\frac{2\mu}{\sigma}B_{m^{-1}s}\right)_{s≥0}
\end{equation*} 
for the Skorokhod topology on $\mathbb{D}(\R_+,\R)$, a convergence that holds jointly with that of~(\ref{eq:cvsimple}). As explained in Section~2.6 of~\cite{miermont}, this can be seen as follows : since $(\varphi(\fl{ns})/n)_{s≥0}$ converges towards a deterministic process, the couple $\left((\varphi(\fl{ns})/n)_{s≥0}, (n^{-1/2}H^1(\fl{ns}))_{s≥0}\right)$ converges in law. Now Skorokhod representation theorem ensures that there exists a probability space where this convergence holds almost surely, and therefore where both convergences of $s\mapsto n^{-1/2}H^1(\fl{ns})$ and of $s\mapsto \varphi(\fl{ns})/n$ hold almost surely. In such a space, the convergence of their composition will hold almost surely, and therefore will hold in distribution.
Finally, Proposition~\ref{prop:vertical} yields 
\begin{equation*}
\Big(\Big|\frac{H^{\ell}(\fl{ns})}{n^{1/2}} - \frac{\mu H^1(\varphi(\fl{ns})}{n^{1/2}}\Big|\Big)_{s≥0}\substack{ \P \\ \longrightarrow \\ n\to\infty} 0,
\end{equation*} for the topology of the convergence over compact sets, thus completing the proof of the theorem. $\hfill □$
\bigskip

\noindent {\it Proof of (ii) and (iii).} (ii) The proof is similar to that of Theorem 1 (ii) of~\cite{miermont}. Denote by $\Gamma_n^1$ the index of the tree in $\F^1$ to which $u^1(n)$ belongs. The definition of $\varphi$ allows us to write for all $n\in\N$, $s≥0$, $\Gamma_{\fl{ns}}=\Gamma^1_{\fl{\varphi(ns)}}$. Proposition~\ref{prop:horizontal} and then Corollary~2.5.1 of~\cite{duquesne-le-gall} applied to $\Gamma^1$ (as $\F^1$ is a monotype Galton--Watson forest) allow us to conclude the proof. \\
(iii) The proof of Corollary~1 of~\cite{miermont} can be applied here, using Theorem~\ref{th:bitype}~(i) and~(ii). $\hfill \square$

\section{Proof of Theorem~\ref{th:multitype} }\label{s:proof2}
\subsection{Change of measure on the multitype Galton--Watson tree}

Let us introduce here the multitype version of what was introduced in Subsection~\ref{subsec:measurechanget1}. Let $(W_n)_{n\in\N}$ be the {\it multitype additive martingale}, where for all $n\in\N$, 
\begin{equation*}
W_n := \sum_{|u|=n}b_{\ty(u)}.  
\end{equation*}
For all $n\in\N$, we let $\cal{F}_n$ be the sigma-algebra generated by the $(u,\ty(u))$ for $u\in\Tb$, $|u|≤n$. Then for all $x_0\in\X$, $(\frac{W_n}{b_{x_0}})$ is a $\P_{x_0}$-martingale for the filtration $(\cal{F}_n)_{n\in\N}$. Indeed, for all $n\in\N$, $W_n$ is obviously $\cal{F}_n$-measurable, and has a finite first moment as $(b_x)_{x\in\X}$ is an ${\bf M}$-right eigenvector. Moreover, 
\begin{align*}
\E_{x_0}\Big[ W_{n+1}\, |\,\mathcal{F}_n \Big]&=\E_{x_0}\Big[ \sum_{|u|=n}\sum_{\parent{v}=u} b_{\ty(v)}\, |\,\mathcal{F}_n \Big]\\
&=\sum_{|u|=n} \E_{\ty(u)}\Big[ \sum_{|v|=1} b_{\ty(v)}\Big]\\
&=\sum_{|u|=n} \sum_{y\in\X}m_{\ty(u),y}b_y\\
&=\sum_{|u|=n}b_{\ty(u)}=W_n, 
\end{align*}
where we used the branching property between lines 2 and 3, and then the fact that $(b_x)_{x\in\X}$ is an ${\bf M}$-right eigenvector. Finally, notice that
\begin{equation*}
\E_{x_0}\Big[ W_0 \Big]=\E_{x_0}\Big[ b_{\ty(\root)} \Big]=b_{x_0}. 
\end{equation*}
Let us introduce a new law $\Ph_{x_0}^*$ on marked trees $(\Tb,\ty)$ with a distinguished path $(w_n)_{n≥0}$ where for any $n≥0$, $w_n$ is at generation $n$. Let $\widehat{\bzeta}=(\widehat{\zeta}_x)_{x\in\X}$ be the probability law of Radon-Nikodym derivative $\sum_{u\in\Tb,|u|=1} b_{\ty(u)}$ with respect to $\bzeta$. More precisely, for any $x\in\X$, if $X\sim \zeta_x$, then $\hat{X}\sim \hzeta_x$ if and only if for any function bounded real-valued function $f$ on $\X^{(\N)}$, 
\begin{equation*}
\E[f(\hat{X})]=\E[|X|f(X)], 
\end{equation*}
where we recall that $|X|$ stands for the length of $X$. We construct $(\Tb,\ty,(w_n)_{n≥0})$ under $\Ph^*_{x_0}$ by induction as follows~: 

\begin{itemize}
\item {\bf Initialisation} \\
Generation $0$ of $\Tb$ is only made up of the root $\root$ of given type $\ty(\root)=x_0$. We set $w_0=\root$. 
\item {\bf Induction} Let $n≥0$. Suppose that the tree and the spine have been built up to generation $n$. The vertex $w_n$ has progeny according to $\hzeta_{\ty(w_n)}$. Other vertices $u$ of generation $n$ have progeny according to $\zeta_{\ty(u)}$. Then, choose a vertex at random among children $u$ of $w_{n}$, each with probability $b_{\ty(u)}/\Big(\sum_{\parent{v}=w_{n}} b_{\ty(v)}\Big)$ and set $w_{n+1}$ as this vertex. 
\end{itemize}

\noindent We denote by $\Ph_{x_0}$ the marginal law of $(\Tb,\ty)$ under this construction, and $\Eh_{x_0}$ the associated expectation. Just as in Subsection~\ref{subsec:measurechanget1}, the following proposition, which is easily deduced from \cite{kurtz-lyons-permantle-peres-multitype}, links $\P_{x_0}$ and $\Ph_{x_0}$~: 

\begin{proposition} {\bf~\cite{kurtz-lyons-permantle-peres-multitype}}\label{prop:markovtype}
\begin{itemize}
\item[\rm (i)] Recall that for any $n≥0$, $\cal{F}_n$ stands for the sigma-algebra generated by the $(u,\ty(u))$ for $u\in\Tb, |u|≤n$. Then $\Ph_{x_0|\cal{F}_n}$ is absolutely continuous with respect to $\P_{x_0|\cal{F}_n}$ and is such that
\begin{equation*}
 \frac{\d \Ph_{x_0}}{\d\P_{x_0}}|_{\cal{F}_n}=\frac{1}{b_{x_0}}W_n. 
\end{equation*}
\item[\rm (ii)] Recall that $\cal{F}_n$ bears no information on $(w_n)_{n≥0}$. Conditionally on $\cal{F}_n$, for all $u\in\Tb$ such that $|u|=n$, 
\begin{equation*}
\Ph^*_{x_0}\Big( w_n=u \, |\, \mathcal{F}_n \Big)=\frac{b_{\ty(u)}}{W_n}
\end{equation*}
\item[\rm (iii)] Under $\Ph^*_{x_0}$, the process $(\phi_k)_{k\in\N}:=(\ty(w_k))_{k\in\N}$ is a Markov chain taking its values in $\X$ with initial state $x_0$, and with transition probabilities denoted by $(p_{x,y})_{x,y\in\X}$, where for all $x,y\in\X$, $p_{x,y}=\frac{b_y}{b_x}m_{x,y}$.  
\end{itemize}\end{proposition}

\noindent Just as in Section~\ref{s:proof1}, as there should be no ambiguity on it, we will indifferently denote $\Ph_{x_0}$ or $\Ph^*_{x_0}$ by $\Ph_{x_0}$, and $\Eh_{x_0}$ their associated expectation. Notice that the Markov chain $(\phi_k)_{k\in\N}$ introduced in (iii) admits an invariant measure $(\pi_x)_{x\in\X}$ where for all $x\in\X$, 
\begin{equation*}
\pi_x=a_x b_x, 
\end{equation*}
and that under $\bf (H_M)$ this measure is finite, thus ensuring that $(\phi_k)_{k\in\N}$ is positive recurrent. Moreover, hypothesis $\bf (H_M)$ implies its irreducibility.  Proposition~\ref{prop:markovtype} yields the \textit{multitype many-to-one lemma} : 
\begin{lemma}\label{lemma:many-to-one-multitype}
For all $n \in \Ns$, $g:\X^n \to \R_+$ a measurable function, $X_n$ a $\cal{F}_n$-measurable random function, 
\begin{equation*}
\E_{x_0}\Big[\sum_{|u|=n}g(\ty(u_1),\ty(u_2),\ldots,\ty(u_n)) X_n \Big]=b_{x_0}\Eh_{x_0}\Big[\frac{1}{b_{\phi_n}}g(\phi_1,\phi_2,\ldots,\phi_n) X_n \Big]. 
\end{equation*}
\end{lemma}
This lemma will be of great use, since thanks to it the study of certain quantities of the multitype Galton--Watson tree can be reduced to that of a simple Markov chain. Let us now prove Proposition~\ref{prop:hypotheses} introduced in Subsection~\ref{s:reduction}.

\subsection{Proof of Proposition~\ref{prop:hypotheses} : Hypothesis $\bf (H_1)$ }\label{s:proofh1}

For all $y\in\X$ and $u\in\Fb$ of $\Tb$, we set :  
\begin{equation*}
\Ni^y_u:=\#\B_u^y, 
\end{equation*}
that is $\Ni_u^y$ is the number of vertices "between" $u$ and $\cal{L}_u^y$, $\cal{L}_u^y$ included. If $u=\root$, we will simply write $\Ni^y$. We want to prove that $\F$ satisfies hypothesis $\bf (H_1)$, which in our case boils down to prove the following proposition~:  

\begin{proposition}\label{prop:EN}
For any $x_0\in\X$, the random variable $\Ni^{x_0}$ has a finite first moment under $\P_x$~; more precisely~: 
\begin{equation*}
\E_{x_0}\Big[ \Ni^{x_0} \Big]=\frac{1}{a_{x_0}}.
\end{equation*}
\end{proposition}

\begin{proof}
For any $y\in\X$, we denote by 
\begin{equation}\label{deftauj}
\tauh_y := \inf\{ k≥1 \st \phi_k=y \}
\end{equation}
the first non-null hitting time of state $y$ by $(\phi_k)_{k≥1}$. Let us show that $\Ni^y$ admits a finite first moment for any $y\in\X$, whatever the type of $\root$~; let $x,y\in\X$, the many-to-one lemma (Lemma~\ref{lemma:many-to-one-multitype}) yields
\begin{align*}
\E_x\Big[ \Ni^y \Big]=& \E_x\Big[ \sum_{u\in\Tb\backslash\{\root\}} \1{\ty(u_1),\ty(u_2),\ldots,\ty(\parent{u})\neq y} \Big] \\
=& \sum_{k≥1}\E_x\Big[ \sum_{|u|=k} \1{\ty(u_1),\ldots,\ty(u_{k-1})\neq y} \Big]\\
=&b_x \sum_{k≥1}\Eh_x\Big[ \frac{1}{b_{\phi_k}} \1{\phi_1,\ldots,\phi_{k-1}\neq y} \Big]=b_x \Eh_x\Big[ \sum_{k=1}^{\tauh_y} \frac{1}{b_{\phi_k}} \Big],
\end{align*}
which yields in the case where $x=y=x_0$, 
\begin{equation*}\E_{x_0}\Big[ \Ni^{x_0} \Big] = b_{x_0} \Eh_{x_0}\Big[ \sum_{k=1}^{\tauh_{x_0}} \frac{1}{b_{\phi_k}} \Big]= b_{x_0} \sum_{z\in\X}\frac{1}{b_z} \frac{\pi_z}{\pi_{x_0}} = \sum_{z\in\X} \frac{a_z}{a_{x_0}}=\frac{1}{a_{x_0}}, 
\end{equation*}
which concludes the proof. In the second equality, we used a classic result on the mean time spent in a given state during a Markovian excursion (we recall that $(\pi_z)_{z\in\X}$ is the invariant measure of $(\phi_k)_{k≥0}$). Then we used the fact that for all $z\in\X$, $\pi_z=a_z b_z$, and then that $\sum_{z\in\X}a_z=1$. 
\end{proof}

\subsection{Proof of Proposition~\ref{prop:hypotheses} : Hypotheses $\bf (H_c)$ and $\bf (H_c^2)$ }\label{s:proofhc}

For all $y\in\X$ and $u\in\Fb$, we let  
\begin{equation*}
\Zi_u^y := \#\cal{L}_u^y
\end{equation*}
be the number of vertices forming $\L_u^y$. If $u=\root$, we will simply write $\Zi^y$. To prove that $\F$ satisfies hypotheses $\bf (H_c)$ and $\bf (H_c^2)$, we just need to show the following proposition. 

\begin{proposition}\label{prop:EZ}
Under $\P_{x_0}$, $\Zi^{x_0}$ has a finite second moment~; more precisely~: 
\begin{equation}\label{valeurs_Z}
\E_{x_0}\Big[ \Zi^{x_0} \Big]= 1 \qquad\textrm{and}\qquad {\bf Var}_{x_0}(\Zi^{x_0}) = \E_{x_0}\Big[ (\Zi^{x_0})^2 \Big]- \E_{x_0}\Big[ \Zi^{x_0} \Big]^2=\frac{\eta^2}{a_{x_0} b_{x_0}^2}. 
\end{equation}
\end{proposition}

\begin{proof} First, let us focus on the first moment of the cardinal of an $\L^y$ stemming from a root of type $x\in\X$~; using the many-to-one lemma we get
\begin{align}\label{Ezij}
\E_x[\Zi^y]&= \E_x\Big[\sum_{u\in\Tb\backslash\{\root\}} \1{\ty(u_1)\neq y, \ldots, \ty(u_{|u|-1})\neq y, \ty(u)=y}\Big] \nonumber \\
&= \sum_{k≥1} \E_x\Big[\sum_{|u|=k} \1{\ty(u_1)\neq y, \ldots, \ty(u_{k-1})\neq y, \ty(u)=y}\Big] \nonumber \\
&= \sum_{k≥1} \Eh_x\Big[ b_x \frac{\1{w_k\in \cal{L}^y}}{b_{\phi_k}} \Big]= \frac{b_x}{b_y}, 
\end{align}
where between the last two lines we used the fact that $(\phi_k)_{k\in\N}$ is positive recurrent. So in the case where $x=y=x_0$ this yields the first equality of~(\ref{valeurs_Z}). Now, let us compute the second moment of the number of vertices forming the first generation of type $y$. Discussing on the generation to which vertices of $\mathcal{L}^y$ belong, we get 
\begin{equation*}
\E_x[(\Zi^y)^2]=\E_x\Big[\Big(\sum_{k≥1}\sum_{|u|=k} \1{u\in{\L^y}}\Big)\times\Zi^y\Big]. 
\end{equation*}
For $k≥1$, let us focus on the general term of the sum. When conditioning on $\cal{F}_k$, it can be written as
\begin{equation*}
\E_x\Big[\Big(\sum_{|u|=k}\1{u\in{\L^y}}\Big)\Zi^y\Big]=\E_x\Big[\Big(\sum_{|u|=k}\1{u\in{\L^y}}\Big)\E_x\Big[\Zi^y \mid \cal{F}_k \Big]\Big]. 
\end{equation*}
Let us apply the many-to-one lemma (Lemma~\ref{lemma:many-to-one-multitype}) at generation $k$ to this expectation, with the setting $X_k=\E_x\Big[\Zi^y \mid \cal{F}_k \Big]$ (which is $\cal{F}_k$-measurable )~; we get
\begin{equation*}
\E_x\Big[\Big(\sum_{|u|=k}\1{u\in{\L^y}}\Big)\Zi^y\Big]=\Eh_x\Big[ b_x\times\frac{1}{b_y}\1{\tauh_y=k} \E_x\Big[\Zi^y \mid \cal{F}_k \Big] \Big]=\frac{b_x}{b_y}\Eh_x\Big[ \1{\tauh_y=k}\Zi^y\Big]. 
\end{equation*}
where we recall that $\tauh_y$ is the first non-null hitting time of $y$ by $(\phi_k)_{k≥0}$. We used the fact that on the event $\{\tauh_y=k\}$, we have $\E_x\left[\Zi^y \mid \cal{F}_k \right]=\Eh_x\left[\Zi^y \mid \cal{F}_k \right]$. Summing over $k≥1$, as $(\phi_k)_{k≥0}$ is recurrent, we finally get a simpler expression of the second moment~: 
\begin{equation*}
\E_x[(\Zi^y)^2]=\frac{b_x}{b_y}\Eh_x[\Zi^y]. 
\end{equation*}

Now, computing this last quantity will require a decomposition more subtle. Under the biased law $\Ph$, $\cal{L}^y$ is made up of 
\begin{itemize}
\item[$\bullet$] the first vertex of the spine $(w_k)_{k≥1}$ being of type $y$, that is $w_{\tauh_y}$, counting for one vertex, 
\item[$\bullet$] the vertices $u$ of type $y$ which are brothers of a $w_k$ for $k≤\tauh_y$, counting for $\sum_{k=1}^{\tauh_y}\sum_{u\in\Omega(w_k)}\1{ \ty(u)=y }$ vertices, 
\item[$\bullet$] the lines $\cal{L}_u^y$ for any brother $u$ of any $w_k$ (with $k≤\tauh_y$) such that $\ty(u)\neq y$, counting for $\sum_{k=1}^{\tauh_y}\sum_{u\in\Omega(w_k)} \1{ \ty(u)\neq y }\Zi_u^y$ vertices, 
\end{itemize}  
where we recall that for $k≥0$, $\Omega(w_k)$ stands for the brothers of $w_k$ ($w_k$ not included). In total, we can write that
\begin{equation}\label{eq:developEhZ}
\Eh_x[\Zi^y]=\Big( 1+\Eh_x\Big[ \sum_{k=1}^{\tauh_y}\sum_{u\in\Omega(w_k)}(\Zi_u^y \1{ \ty(u)\neq y } +1\times\1{ \ty(u)=y }) \Big] \Big) 
\end{equation}
after this decomposition along the spine. Conditioning with respect to $\sigma((w_k)_{k\in\N},(\Omega(w_k))_{k\in\N})$ and using the fact that $\E_x[\Zi^y]=\frac{b_x}{b_y}$, this last expectation is equal to 
\begin{align}\label{eq:EiZj2debut}
\Eh_x\Big[ \sum_{k=1}^{\tauh_y}\sum_{u\in\Omega(w_k)}(\Zi_u^y \1{ \ty(u)\neq y } +1\times\1{ \ty(u)=y }) \Big]&=\Eh_x\Big[ \sum_{k=1}^{\tauh_y}\sum_{u\in\Omega(w_k)} \frac{b_{\ty(u)}}{b_y} \Big] \nonumber\\
&=\frac{1}{b_y}\Eh_x\Big[ \sum_{k=0}^{\tauh_y-1} \Big( (\sum_{\vec{u}=w_k} b_{\ty(u)}) - b_{\phi_{k+1}} \Big) \Big]\nonumber\\
&=\frac{1}{b_y}\Eh_x\Big[ \sum_{k=0}^{\tauh_y-1} \Eh_{\phi_k}\Big[ \Big(\sum_{|u|=1} b_{\ty(u)}\Big) - b_{\phi_{1}} \Big] \Big], 
\end{align}
where we used the branching property on each $w_k$ for $0≤k≤\tauh_y-1$. Discussing on the type of $w_k$ in the inner expectation, this can be written as
\begin{equation*}
\Eh_x\Big[ \sum_{k=1}^{\tauh_y}\sum_{u\in\Omega(w_k)}(\Zi_u^y \1{ \ty(u)\neq y } +1\times\1{ \ty(u)=y })\Big] =\frac{1}{b_y}\sum_{z\in\X} \Eh_x\Big[ \sum_{k=0}^{\tauh_y-1} \1{\phi_k=z} \Big] \Eh_z\Big[ \Big(\sum_{|u|=1} b_{\ty(u)}\Big) - b_{\phi_1} \Big]. 
\end{equation*}
Let us clarify the term $\Eh_z\Big[ \Big(\sum_{|u|=1} b_{\ty(u)}\Big) - b_{\phi_1} \Big]$ for any $z\in\X$. Noticing that \mbox{$\Eh_z\Big[\1{u=w_1}\mid \cal{F}_1\Big]$}$=\frac{b_{\ty(u)}}{\sum_{|u|=1} b_{\ty(u)}}$ as explained in the construction of $\Tbh$, and that ${\frac{\d \Ph_z}{\d\P_z}}|_{\cal{F}_1}=\frac{\sum_{|u|=1}b_{\ty(u)}}{b_z}$, we get
\begin{align*}
\Eh_z\Big[ \Big( \sum_{|u|=1} b_{\ty(u)} \Big) - b_{\phi_1} \Big]=& \Eh_z\Big[\sum_{|u|=1}\Big( b_{\ty(u)} - b_{\ty(u)}\1{ u=w_1 } \Big) \Big]\\
=& \Eh_z\Big[\sum_{|u|=1}\Big( b_{\ty(u)} - b_{\ty(u)}\times \frac{b_{\ty(u)}}{\sum_{|u|=1}b_{\ty(u)}} \Big) \Big]\\
=&\E_z\Big[\frac{\sum_{|u|=1}b_{\ty(u)}}{b_z} \Big( \sum_{|u|=1}\Big( b_{\ty(u)} - b_{\ty(u)}\times \frac{b_{\ty(u)}}{\sum_{|u|=1}b_{\ty(u)}} \Big)\Big) \Big]\\
=&\frac{1}{b_z}\E_z\Big[ \Big( \sum_{|u|=1} b_{\ty(u)}\Big)^2 - \sum_{|u|=1}(b_{\ty(u)})^2 \Big].  
\end{align*}
Discussing on the type of $u$ in this last expectation, we get
\begin{align*}
\Eh_z\Big[ \Big( \sum_{|u|=1} b_{\ty(u)}\Big) - b_{\phi_1} \Big]=&\frac{1}{b_z} \sum_{{x'},{y'}\in\X} b_{x'} b_{y'} \E_z\Big[ \Big( \sum_{|u|=1} \1{ \ty(u)={x'} } \Big)\Big( \sum_{|u|=1} \1{ \ty(u)={y'} } \Big) - \delta_{x',y'}\sum_{|u|=1}\1{ \ty(u)=x' } \Big]\\
=&\frac{1}{b_z} \sum_{{x'},{y'}\in\X} b_{x'} Q_{x',y'}^z b_{y'}, 
\end{align*}
so plugging this in~(\ref{eq:EiZj2debut}), and then plugging~(\ref{eq:EiZj2debut}) in~(\ref{eq:developEhZ}) yields
\begin{equation*}\label{EiZj2}
{\E}_x[(\Zi^y)^2]=\frac{b_x}{b_y}\Big( 1+\sum_{z\in\X} \Eh_x\Big[ \sum_{k=0}^{\tauh_y-1} \1{ \phi_k=z } \Big] \frac{1}{b_z b_y} \sum_{x',y'\in\X}b_{x'} Q_{x',y'}^z b_{y'} \Big).
\end{equation*}
Now if $x=y=x_0$, we finally get 
\begin{align*}\label{Zi2_rec}
{\E}_{x_0}[{(\Zi^{x_0}})^2]&=1+\sum_{z\in\X} \frac{\pi_z}{\pi_{x_0}} \frac{1}{b_z b_{x_0}} \sum_{{x'},{y'}\in\X}b_{x'} Q_{{x'},{y'}}^z b_y' \nonumber \\
&=1+\frac{1}{a_{x_0} {b_{x_0}}^2}\sum_{z\in\X} \sum_{x',y'\in\X} a_z b_{x'} Q_{x',y'}^z b_{y'} \nonumber = 1+ \frac{\eta^2}{a_{x_0} {b_{x_0}}^2}. 
\end{align*}

\noindent Thus, the variance of our leafed Galton--Watson tree with edge lengths $\T$ is finite under $\bf (H_Q)$ and computed as~: 
\begin{equation*}
{\bf Var}_{x_0}(\Zi^{x_0})={\E}_{x_0}[({\Zi^{x_0}})^2]-{\E}_{x_0}[{\Zi^{x_0}}]^2=\frac{\eta^2}{a_{x_0} {b_{x_0}}^2}, 
\end{equation*}
which concludes~(\ref{valeurs_Z}), and the proof of Propositions~\ref{prop:EZ} and~\ref{prop:hypotheses}. 
\end{proof}

\subsection{ Conclusion of the proof of Theorem~\ref{th:multitype}} 
Now, we can conclude the proof of Theorem~\ref{th:multitype}. Indeed, since $\F$ satisfies hypothesis $\bf (H)$, one can apply Theorem~\ref{th:bitype} to $(H^{\ell}(n))_{n\in\N}$ : under $\P_{x_0}$,
\begin{equation}\label{eq:cvfinalenot}
\Big(\frac{H^{\ell}(\fl{ns})}{\sqrt{n}}\Big)_{s≥ 0} \substack { \Longrightarrow \\ n\to\infty} \Big(\frac{2\mu}{\sigma}|B_{{m}^{-1}s}|\Big)_{s≥ 0},
\end{equation}
where the convergence holds in law for the Skorokhod topology on the space $\mathbb{D}(\R_+,\R)$ of c\`adl\`ag functions, and where $B$ is a standard Brownian motion. Here, 
\begin{itemize}
\item $\begin{aligned}[t]
\mu=\E\Big[ \sum_{u\in\T,\,|u|=1,\, \ty(u)=1} \ell(u) \Big]&=\E_{x_0}\Big[ \sum_{u\in\Tb,\,u\in\cal{L}^{x_0}} |u| \Big] &\\
&= \sum_{k≥1} \E_{x_0}\Big[ \sum_{u\in\Tb,\, |u|=k} |u| \1{u\in\cal{L}^{x_0}} \Big]  &\\
&=\sum_{k≥1}b_{x_0}\Eh_{x_0}\Big[\frac{1}{b_{\phi_k}} |w_k|\1{k=\tauh_{x_0}} \Big] &\\
&=\Eh_{x_0}[\tauh_{x_0}]=\frac{1}{a_{x_0}b_{x_0}}, &
\end{aligned}$ \\
since $(\pi_x)_{x\in\X}=(a_x b_x)_{x\in\X}$ is the invariant measure of $\phi$. We used Lemma~\ref{lemma:many-to-one-multitype} between lines 2 and 3.

\item $\begin{aligned}
&\sigma^2={\bf Var}\Big[ \sum_{u\in\T,|u|=1,\ty(u)=1} 1 \Big]={\bf Var}_{x_0}\Big[ \Zi^{x_0} \Big]=\frac{\eta^2}{a_{x_0}{b_{x_0}}^2},&
\end{aligned}$ \\
by Proposition~\ref{prop:EZ}. 
\item $\begin{aligned}[t]
m=\E\Big[  \sum_{u\in\T|u|=1} 1 \Big]=\E_{x_0}\Big[ \Ni^{x_0} \Big]=\frac{1}{a_{x_0}}, 
\end{aligned}$ \\
by Proposition~\ref{prop:EN}. 
\end{itemize}
Plugging this into~(\ref{eq:cvfinalenot}), and using the fact that for all $n\in\N$, $H^{\ell}(n)=|u_{\Fb}(n)|$ (as specified in Proposition~\ref{prop:egaliteprocess}), we finally get
\begin{equation*}
\Big(\frac{|u_{\Fb}(\fl{ns})|}{\sqrt{n}}\Big)_{s≥ 0} \substack { \Longrightarrow \\ n\to\infty} \Big(\frac{2}{\eta}|B_s|\Big)_{s≥ 0},
\end{equation*}
which is what we wanted to prove Theorem~\ref{th:multitype}~(i). The proof of Theorem~\ref{th:multitype}~(ii) and~(iii) is now similar to that of Theorem~\ref{th:bitype}~(ii) and~(iii). $\hfill \square$ 

\section{An application of Theorem~\ref{th:multitype} to random laminations}\label{s:application}

In~\cite{curien-peres}, N.~Curien and Y.~Peres study certain aspects of the random laminations of the disk, and this study is reduced to that of a multitype Galton--Watson tree $\Tb$ with types taking values in $\llbracket 4;+\infty \llbracket$. Vertices $u$ of type $m≥4$ give progeny the following way~: choose $m'\in\brint{0;m}$ uniformly at random, and if $m'≥3$ then $u$ has a child of type $1+m'$, if $m'≤m-3$ then $u$ has a child of type $1+m-m'$ (note that if these two conditions are satisfied $u$ gives birth to two children). We propose an alternative proof of Theorem 1.1 of~\cite{curien-peres}, simply applying Theorem~\ref{th:multitype} (iii). 
\begin{theorem}
Under $\P_4$, population at generation $n$ denoted by $Z_n$ is such that
\begin{equation*}
\E_{4}\Big[ Z_n \Big]\sto{n\to\infty} \frac{4}{e^2-1}. 
\end{equation*}
Moreover, the probability that $Z_n\neq 0$ is such that 
\begin{equation*}
\P_4\Big( Z_n\neq 0 \Big)\sto[\sim]{n\to\infty} \frac{5(e^2-1)^2}{8n}. 
\end{equation*}
\end{theorem}

\begin{proof}
In this proof, we will use the notation of the previous sections. A computation leads to a mean matrix ${\bf M}=(m_{i,j})_{i,j≥4}$ where for $i,j≥4$, $m_{i,j}=\frac{2}{i+1}\1{j≤i+1}$. That is $\bf M$ is such that~: 
\begin{equation*}
{\bf M}= \begin{pmatrix}
  \frac{2}{5} & \frac{2}{5} & 0& 0 & 0 & \cdots &  \\
  \frac{2}{6} & \frac{2}{6} & \frac{2}{6} & 0& 0 & \cdots &  \\
  \frac{2}{7} & \frac{2}{7} & \frac{2}{7} & \frac{2}{7}& 0 & \cdots &  \\
  
  \vdots  & \vdots  & \vdots &\vdots& \vdots & \ddots & 
 \end{pmatrix}. 
\end{equation*}
Following conditions of Subsection~\ref{s:intromulti}, we are looking for a left eigenvector $(a_i)_{i≥4}$ and a right eigenvector $(b_i)_{i≥4}$ such that for $i≥4$~: 
\begin{equation*}
(1-\frac{1}{i+2})b_{i}-b_{i+1}+\frac{2}{i+2}b_{i+2}=0\;\textrm{ and }\; \frac{2}{i+1}a_i-a_{i+1}+ a_{i+2}=0. 
\end{equation*}
and with initial conditions $a_4=a_5$ and $b_5=\frac{3}{2}b_4$. A computation indicates that these equations are satisfied by 
\begin{equation*}
(b_i)_{i≥4}=(\frac{2}{e^2-1}(i-2))_{i≥4}\textrm{ and } (a_i)_{i≥4}=(\frac{2^{i-3}(i-3)}{(i-1)!})_{i≥4}, 
\end{equation*}
vectors which satisfy $\sum_{i≥4} a_i=1$ and $\sum_{i≥4} a_ib_i=1$. Thus, the multitype Galton--Watson tree $\Tb$ here satisfies hypothesis $\bf (H_M)$. Moreover, a computation gives for $i,j,k≥4$, 
\begin{equation*}
Q_{i,j}^k=\frac{2}{k+1}\1{i=k+2-j},
\end{equation*}
which yields 
\begin{equation*}
\eta^2=\frac{16}{5(e^2-1)^2}<\infty,
\end{equation*}
and so $\bf (H_Q)$ is also satisfied. We now want our tree to satisfy $\bf (H_R^{alt})$ (introduced in the appendix); the transition probabilities of the resulting Markov chain $(\phi_{n})_{n≥4}$ are given by 
\begin{equation*}
p_{i,j}=\frac{2(j-2)}{(i-2)(i+1)}\1{4≤j≤i+1}
\end{equation*}
for $i,j≥4$. Let us set for all $n≥4$, $V(n)=\beta^n$ for a any $\beta>1$. We notice that $(\phi_{n})_{n≥4}$ satisfies condition~(\ref{eq:lyapunov}) with $(V(n))_{n≥4}$ dominating $(\frac{1}{b_n})_{n≥4}$ for $n$ large enough. Thus the tree $\Tb$ satisfies hypothesis $\bf (H_R^{alt})$. 
Anyway, we get, applying Lemma~\ref{lemma:many-to-one-multitype}, 
\begin{equation*}
\E_{4}\Big[ Z_n \Big]=\E_{4}\Big[ \sum_{|u|=n}1 \Big]=b_4\Eh\Big[ \frac{1}{b_{\phi_n}} \Big].
\end{equation*}
The Markov chain $(\phi_n)_{n≥4}$ being irreducible, aperiodic and having $(\pi_i)_{i≥4}=(\frac{2}{e^2-1}\frac{2^{i-3}(i-3)(i-2)}{(i-1)!})_{i≥4}$  for invariant measure, we get 
\begin{equation*}
\E_{4}\Big[ Z_n \Big]=b_4\Eh\Big[\frac{1}{b_{\phi_n}} \Big]\sto{n\to\infty} b_4\sum_{i≥4}\frac{1}{b_i}\pi_i=\frac{4}{e^2-1}. 
\end{equation*}
Now, $\Tb$ satisfies the conditions of Theorem~\ref{th:multitype}, and then (iii) yields
\begin{equation*}
\P(Z_n\neq0)\sto[\sim]{n\to\infty} \frac{2}{\eta^2 n}=\frac{5(e^2-1)^2}{8}\times\frac{1}{n}, 
\end{equation*}
which completes the proof. 
\end{proof}
\appendix
\section{Appendix}

Conditions $\bf (H_R^{x_0})$ may be not convenient to check. In this appendix, we propose a more practical hypothesis. We recall the statement of hypothesis $\bf(H_R^{x_0})$ for ${x_0}\in\X$~: 
\begin{equation*}
{\bf (H_R^{x_0})}\begin{cases}\parbox{0.9\textwidth}{\begin{itemize} 
\item[$\bullet$] $y^2\P_{x_0}\Big( \max\{ |u| \st u\in\Tb,\: \ty(u_1),\ldots,\ty(u_{|u|})\neq {x_0} \}>y \Big)\sto{y\to\infty} 0$, 
\item[$\bullet$] $y^2\E_{x_0}\Big( \sum_{|u|>y}\1{ \ty(u_1),\ldots,\ty(u_{|u|-1})\neq x_0,\:\ty(u)={x_0}} \Big)\sto{y\to\infty} 0$. 
\end{itemize}
}\end{cases}
\end{equation*}
We propose an alternative hypothesis to $\bf(H_R^{x_0})$ : 
\begin{equation*}
{\bf (H_R^{alt})}\begin{cases}\parbox{0.9\textwidth}{ There exists a function $V:\X\to[1;\infty)$, a finite set $C\subset\X$ and a constant $\beta>0$, such that for all $x\in\X\setminus C$,
\begin{equation}\label{eq:lyapunov}
\sum_{y\in \X} p_{x,y} V(y) ≤ (1-\beta) V(x)
\end{equation}
the function $V$ being such that for all $x\in\X\setminus C$, $\frac{1}{b_x}≤V(x)$.
}\end{cases}
\end{equation*}

We recall that $(b_x)_{x\in\X}$ is the left eigenvector introduced in $\bf (H_M)$, and that $(p_{x,y})_{x,y\in\X}$ are the transition probabilities of $(\phi_k)_{k\in\X}$ introduced in Proposition~\ref{prop:markovtype}. A Markov chain satisfying condition~(\ref{eq:lyapunov}) is said to be {\it geometric ergodic} \cite{meyn-tweedie}. Notice that any Markov chain on a finite space satisfies such a condition, as we just have to choose $C=\X$ and any $V≥1$. Hence, hypothesis $\bf (H_R^{alt})$ is always satisfied if $\X$ is finite, and (according to the proposition below) so is $\bf (H_R^{x_0})$ for any $x_0\in\X$. The notion of geometric ergodicity is well discussed in Chapter 15 of~\cite{meyn-tweedie}. 

\begin{proposition}
Hypothesis $\bf (H_R^{alt})$ implies hypothesis $\bf (H_R^{x_0})$ for any $x_0\in\X$. 
\end{proposition}
\begin{proof}

Suppose $\bf (H_R^{alt})$ is satisfied. Set $x_0\in\X$. Using Markov's inequality, notice that the first condition of $\bf (H_R^{x_0})$ would be satisfied if 
\begin{equation*}
\E_{x_0}\Big[ \big(\max\{ |u| \st \ty(u_1),\ldots,\ty(u_{|u|})\neq {x_0} \}\big)^2 \Big]<\infty. 
\end{equation*}
But, using Lemma~\ref{lemma:many-to-one-multitype}, we get
\begin{align}\label{eq:cond1}
\E_{x_0}\Big[ \big(\max\{ |u| \st \ty(u_1),\ldots,\ty(u_{|u|})\neq {x_0} \}\big)^2 \Big]&≤\E_{x_0}\Big[ \sum_{u\in\Tb} |u|^2 \1{\ty(u_1),\ldots,\ty(u_{|u|})\neq {x_0}} \Big]\nonumber \\
&=b_{x_0}\Eh_{x_0}\Big[ \sum_{k≥0} \frac{|w_k|^2}{b_{\phi_k}} \1{\phi_1,\ldots,\phi_k\neq {x_0}} \Big] \nonumber \\
&=b_{x_0}\Eh_{x_0}\Big[ \sum_{k=0}^{\tauh_{x_0}-1} \frac{k^2}{b_{\phi_k}}\Big],
\end{align}
so if this last quantity is finite, then the first condition of $\bf (H_R^{x_0})$ is satisfied. \\
Notice also that, using Lemma~\ref{lemma:many-to-one-multitype} again, the second condition of $\bf (H_R^{x_0})$ is equivalent to 
\begin{equation*}
y^2\Eh_{x_0}\Big[ \1{\tauh_{x_0}>y}\Big]\sto{y\to\infty} 0, 
\end{equation*}
a condition that would be satisfied if
\begin{align}\label{eq:cond2}
\Eh_{x_0}\Big[ {\tauh_{x_0}}^2 \Big]<\infty. 
\end{align}
Now, notice that hypothesis $\bf (H_R^{alt})$ is such that our Markov chain satisfies condition $(V4)$ of~\cite{meyn-tweedie} (see Subsection~15.2.2 p.376) with the setting $\beta=d-1$. Theorem 15.2.6 of~\cite{meyn-tweedie} with the setting $A=\{x_0\}$ then ensures that $\{x_0\}$ is $V$-geometrically regular (in the sense of the definition given in Subsection~15.2.1 p.\ 373 of~\cite{meyn-tweedie}). In particular, there exists $r>1$ such that
\begin{equation}\label{eq:sumregularfinite}
\Eh_{x_0}\Big[ \sum_{k=0}^{\tauh_{x_0}-1} V(\phi_k)r^k \Big]<\infty. 
\end{equation} 
Since $V$ is greater than $1$, this implies the finiteness of some exponential moments of $\tauh_{x_0}$, and therefore~(\ref{eq:cond2}) is satisfied. Moreover, since $\frac{1}{b_.}≤V(.)$ outside of $C$, 
\begin{align*}
\Eh_{x_0}\Big[ \sum_{k=0}^{\tauh_{x_0}-1} \frac{k^2}{b_{\phi_k}}\Big]&≤\Eh_{x_0}\Big[ \sum_{k=0}^{\tauh_{x_0}-1} \1{\phi_k\notin C}V(\phi_k){k^2}\Big]+(\max_{x\in C}\frac{1}{b_x})\Eh_{x_0}\Big[ \sum_{k=0}^{\tauh_{x_0}-1} \1{\phi_k\in C}{k^2}\Big] \\
&≤\Eh_{x_0}\Big[ \sum_{k=0}^{\tauh_{x_0}-1} (\max_{x\in C} \frac{1}{b_x}+V(\phi_k)){k^2}\Big]
\end{align*}
which is finite according to equation (\ref{eq:sumregularfinite}) ($\max_{x\in C}\frac{1}{b_x}$ being finite because $C$ is finite), thus ensuring the finiteness of~(\ref{eq:cond1}). Hypothesis $\bf (H_R^{x_0})$ is therefore satisfied. 
\end{proof}

\bigskip

\bigskip

\noindent {\bf Acknowledgement :} I thank the referee for his/her numerous and helpful comments. I thank also my advisor Elie Aïdékon for his help and his guidance all along the elaboration of this article.

\end{document}